\documentclass[matprg]{svjour3}
\usepackage{amsmath}
\usepackage{amsfonts}
\usepackage{color}
\usepackage{srcltx}

\usepackage{amsmath}
\usepackage{amsfonts}
\usepackage{latexsym}
\usepackage{xcolor}

\newtheorem{assumption}{Assumption}

\newcommand{\beq}{\begin{equation}}
\newcommand{\eeq}{\end{equation}}
\newcommand{\beqa}{\begin{eqnarray}}
\newcommand{\eeqa}{\end{eqnarray}}
\newcommand{\beqas}{\begin{eqnarray*}}
\newcommand{\eeqas}{\end{eqnarray*}}
\newcommand{\bi}{\begin{itemize}}
\newcommand{\ei}{\end{itemize}}
\newcommand{\ba}{\begin{array}}
\newcommand{\ea}{\end{array}}

\newcommand{\nn}{\nonumber}

\def\endproof{{\ \hfill\hbox{%
      \vrule width1.0ex height1.0ex
    }\parfillskip 0pt}\par}

\def\eqnok#1{(\ref{#1})}
\def\argmin{{\rm argmin}}

\def\Argmin{{\rm Argmin}}

\def\vgap{\vspace*{.1in}}
\def\endproof{{\ \hfill\hbox{%
      \vrule width1.0ex height1.0ex
    }\parfillskip 0pt}\par}

\def\setU{{X}}

\newcommand{\bbe}{\Bbb{E}}

\newcommand{\bbr}{\Bbb{R}}
\def\w{\omega}

\def\endproof{{\ \hfill\hbox{%
      \vrule width1.0ex height1.0ex
    }\parfillskip 0pt}\par}

\def\cD{{\cal D}}

\def\cX{X}
\def\cY{ Y}

\def\lb{{\rm lb}}
\def\ub{{\rm ub}}

\def\C{{C}}

\def\levelset{{\cal E}}
\def\lov{{\vartheta}}
\def\ABLGAP {{\cal G}_{ABL}}
\def\APLGAP {{\cal G}_{APL}}
\def\USLGAP {{\cal G}_{USL}}

\title{
Bundle-Level Type Methods Uniformly Optimal \\
for Smooth and Nonsmooth Convex Optimization
\footnote{{\bf The paper is a combined version of the two manuscripts
previously submitted to {\sl Mathematical Programming}, namely: ``Bundle-type methods uniformly optimal for smooth and nonsmooth convex optimization''
and ``Level methods uniformly optimal for composite and  structured nonsmooth convex optimization''.}}
\thanks{The author of this paper was partially supported by
    NSF grant CMMI-1000347, 
    ONR grant N00014-13-1-0036 and
    NSF CAREER Award CMMI-1254446.}
}
\author{
    Guanghui Lan
    \thanks{Department of Industrial and Systems Engineering,
    University of Florida, Gainesville, FL 32611
       (email: {\tt glan@ise.ufl.edu}). }
}

\begin{document}

\maketitle

\begin{abstract}
The main goal of this paper is to develop uniformly optimal first-order
methods for convex programming (CP). By uniform optimality we mean
that the first-order methods themselves do not require the input of any problem
parameters, but can still achieve the best possible iteration complexity bounds.
By incorporating a multi-step acceleration scheme into the well-known bundle-level
method, we develop an accelerated bundle-level (ABL) method, and show that it can achieve 
the optimal complexity for solving a general class of black-box CP problems 
without requiring the input of any smoothness information, such as, whether the problem is smooth,
nonsmooth or weakly smooth, as well as the specific values of Lipschitz constant
and smoothness level. We then develop a more practical, restricted memory version of
this method, namely the accelerated prox-level (APL) method. We investigate the generalization
of the APL method for solving
certain composite CP problems and 
an important class of saddle-point problems recently studied by Nesterov [Mathematical Programming, 103 (2005), pp 127-152]. 
We present promising numerical results for these new bundle-level methods 
applied to solve 
certain classes of semidefinite programming (SDP) and stochastic programming (SP) problems.
\vspace{.1in}


\end{abstract}

\vspace{0.1cm}
\setcounter{equation}{0}
\section{Introduction} \label{sec_intro}
Consider the convex programming (CP) 
\beq \label{cp}
f^* := \min_{x \in X} f(x),
\eeq
where $X$ is a convex compact set and $f: X \to \bbr$
is a closed convex function. In the classic black-box setting,
$f$ is represented by a first-order oracle
which, given an input point $x \in X$, returns $f(x)$ and $f'(x) \in \partial f(x)$, 
where $\partial f(x)$ denotes the subdifferential of $f$ at $x\in X$.

If $f$ is a general nonsmooth Lipschitz continuous convex function, then, by
the classic complexity theory for CP~\cite{nemyud:83}, the number of 
calls to the first-order oracle for finding an {\sl $\epsilon$-solution}
of \eqnok{cp} (i.e., a point $\bar x \in X$ s.t. $f(\bar x) - f^* \le \epsilon$),
cannot be smaller than ${\cal O}(1/\epsilon^2)$ when $n$ is sufficiently large.
This lower complexity bound can be achieved, for example, by the simple
subgradient descent or mirror descent method~\cite{nemyud:83}.
If $f$ is a smooth function with Lipschitz continuous gradient,
Nesterov in a seminal work~\cite{Nest83-1} presented an algorithm
with the iteration complexity bounded by ${\cal O}( 1 / \epsilon^\frac{1}{2})$, which,
by \cite{nemyud:83}, is also optimal for smooth convex 
optimization if $n$ is sufficiently large. Moreover,
if $f$ is a weakly smooth function with H{\"o}lder
continuous gradient, i.e.,
$\exists$ constants $\rho \in (0, 1)$ and $M > 0$ such that
$
\|\nabla f(x) - \nabla f(y) \|_* \le M \|x - y\|^{\rho},
 \forall \, x, y \in X,
$
then the optimal iteration complexity bound is 
given by ${\cal O}(1 / \epsilon^\frac{2}{1+3\rho})$ 
(see \cite{NemNes85-1,Nest88-1,DeGlNe10-1}).

To accelerate the solutions of large-scale CP problems, 
much effort has recently been directed to exploiting the problem's structure,
in order to identify possibly some new classes of CP problems with stronger convergence
performance guarantee. One such example is given by the composite CP problems
with the objective function given by $f(x) = \Psi(\phi(x))$.
Here $\Psi$ is a relatively simple nonsmooth convex function
such as $\Psi(\cdot) = \|\cdot\|_1$ or $\Psi(\cdot) = \max \left\{y_1,
\ldots, y_k\right\}$ (see Subsection~\ref{sec-comp} for more examples)
and $\phi$ is a $k$-dimensional vector function, see 
\cite{Nest89,Nest04,Nest07-1,Nem94,LewWri09-1,Lan10-3,GhaLan12-2a,GhaLan10-1b}.
In most of these studies, the components of $\phi$ are assumed to
be smooth convex functions. In this case, the iteration complexity 
can be improved to ${\cal O}(1/\epsilon^\frac{1}{2})$ by 
properly modifying Nesterov's optimal smooth method, see for example,
\cite{Nest04,Nest07-1,Nem94}. \textcolor{blue}{It should be noted that these optimal first-order
methods for general composite CP problems are in a sense ``conceptual'' since they require
the minimization of the summation of a prox-function together with the composition
of $\Psi$ with an affine transformation~\cite{Nest04}.}
More recently, Nesterov~\cite{Nest05-1} studied a class of nonsmooth
convex-concave saddle point problems, where the objective function $f$,
in its basic form, is given by 
\[
f(x) = \max_{y \in Y} \langle Ax, y \rangle.
\]
Here $Y \subseteq \bbr^m$ is a convex compact set and $A$ denotes a linear operator
from $\bbr^n$ to $\bbr^m$. Nesterov shows that $f$ can be closely approximated by 
a certain smooth convex function and that the iteration complexity for 
solving this class of problems can be improved to ${\cal O}(1/\epsilon)$.
It is noted in \cite{jnt08} that this bound is unimprovable, for example,
if $Y$ is given by a Euclidean ball and the algorithm
can only have access to $A$ and $A^*$ (the adjoint operator of $A$). 
These problems were
later studied in \cite{Nem05-1,Nest05-2,AuTe06-1,Nest06-1,pena08-1,LaLuMo11-1} and
found many interesting applications, for example, in
\cite{dbg08-1,Lu09-1,BeBoCa09-1}.

\textcolor{blue}{The advantages of the aforementioned optimal first-order methods (e.g.,
subgradient method or Nesterov's method) mainly consist of their optimality, simplicity 
and cheap iteration cost.} However, these methods might have some shortcomings in
that each method is designed for solving a particular subclass of CP problems
(e.g., smooth or nonsmooth). In particular, nonsmooth CP algorithms usually cannot make use of 
local smoothness properties that a nonsmooth instance might have, while it is well-known 
that Lipschitz continuous functions are differentiable almost everywhere within its domain.
\textcolor{blue}{On the other hand, although it has been shown recently in \cite{Lan10-3} that Nesterov's 
method, which was originally designed for solving smooth CP problems, is also optimal for nonsmooth optimization
when employed with a properly specified stepsize policy (see also \cite{DeGlNe10-1} for
a more recent generalization to weakly smooth CP problems), one still needs to determine some 
smoothness properties of $f$ (e.g., whether $f$ is smooth or not, i.e., $\rho =1$ or $0$, and the specific value of $M$),
as well as some other global information (e.g., $D_X$ and in some cases, the number of iterations $N$),
before actually applying these generalized algorithms.}
Since these parameters describe the structure of CP problems over a global scope, these types of
algorithms are still inherently worst-case oriented.

To address these issues, we propose to study the so-called 
{\sl uniformly optimal first-order methods}. 
The key difference between uniformly optimal
methods and existing ones is that they can achieve the best possible complexity for solving different 
subclasses of CP problems, but require little (preferably no) structural information for their implementation. 
To this end, we focus on a different type of first-order methods,
namely: the bundle-level (BL) methods. Evolving from the well-known bundle 
methods~\cite{Kiw83-1,Kiw90-1,Lem75}, 
the BL method was first proposed by Lemar\'{e}chal et al.~\cite{LNN}
in 1995. 
In contrast to subgradient or mirror
descent methods for nonsmooth CP, the BL method can achieve 
the optimal ${\cal O} (1 / \epsilon^2)$ iteration complexity
for general nonsmooth CP without requiring the input of any problem parameters. Moreover, 
the BL method and their certain \textcolor{blue}{``restricted-memory''} variants~\cite{BenNem00,BenNem05-1,Rich07-1} often
exhibit significantly superior practical performance to subgradient or mirror descent methods.
However, to the best of our knowledge, the study on BL methods
has so far been focused on general nonsmooth CP problems only.

Our contribution in this paper mainly consists 
of the following aspects. 
Firstly, we consider a general class of black-box CP problems in the form of \eqnok{cp}, where 
$f$ satisfies
\textcolor{blue}{
\beq \label{smoothness}
f(y) - f(x) - \langle f'(x), y - x\rangle \le \frac{M}{1+\rho} \|y - x\|^{1+\rho}, \ \ \forall x, y \, \in X.
\eeq}
for some $M > 0$, $\rho \in [0,1]$ and $f'(x) \in \partial f(x)$. Clearly,
this class of problems cover nonsmooth ($\rho = 0$),
smooth ($\rho=1$) and weakly smooth ($\rho \in (0,1)$) CP problems (see for example,
p.22 of \cite{Nest04} for the standard arguments used in smooth and
weakly smooth case, and Lemma~2 of \cite{Lan10-3} for a related result in the nonsmooth case). 
By incorporating into the BL method a multi-step acceleration
scheme that was first used by
Nesterov~\cite{Nest83-1} and later in \cite{AuTe06-1,Lan10-3,LaLuMo11-1,Nest04,Nest05-1}
to accelerate gradient type methods for solving smooth CP problems, we present a new BL-type algorithm, namely: the accelerated
bundle-level (ABL) method. We show that the iteration complexity of the
ABL method can be bounded by
\[
{\cal O}\left\{ \left( \frac{ M D_X^{1+\rho}}{\epsilon}\right)^\frac{2}{1+3 \rho} \right\}.
\]
Hence, the ABL method is optimal for 
solving not only nonsmooth, but also smooth and weakly smooth CP problems.
More importantly, this method does not require the input of any smoothness information, such as whether a problem is
smooth, nonsmooth or weakly smooth, and the specific values of problem parameters $M$, $\rho$
and $D_X$. To the best of our knowledge, this is the first time that uniformly optimal
algorithms of this type have been proposed in the literature.

Secondly, one problem for the ABL method is that, as the algorithm proceeds,
its subproblems become more difficult to solve.
As a result, each iteration of the ABL method becomes
computationally more and more expensive. To remedy this issue, we present
a restricted memory version of this method, namely: the accelerated prox-level (APL) method, 
and demonstrate that it can also uniformly achieve the optimal complexity
for solving any black-box CP problems. In particular, each iteration of the APL method requires
the projection onto the feasible set $X$ coupled with a few extra
linear constraints, and the number of such linear constraints  
can be fully controlled (as small as $1$ or $2$). 
The basic idea of this improvement is to incorporate a novel rule due to Kiwiel~\cite{Kiw95-1}
(later studied by Ben-tal and Nemirovski~\cite{BenNem00,BenNem05-1})
for updating the lower bounds and prox-centers. In addition,
non-Euclidean prox-functions can be employed
to make use of the geometry of the feasible
set $X$ in order to obtain (nearly) dimension-independent iteration complexity.

Thirdly, we investigate the generalization of the APL method
for solving certain classes of composite and structured
nonsmooth CP problems. In particular, we show that with
little modification, the APL method is optimal
for solving a class of generalized composite 
CP problems with the objective given by $f(x) = \Psi(\phi(x))$.
Here $\phi_i(x)$, $i \ge 1$, can be a mixture of smooth, nonsmooth, weakly smooth or affine components.
Such a formulation covers a wide range of CP problems, including
the nonsmooth, weakly smooth, smooth, minimax, and regularized CP problems
(see Subsection~\ref{sec-comp} for more discussions). The APL method
can achieve the optimal iteration complexity for solving this class of composite problems without requiring any global 
information on the inner functions, 
such as the smoothness level and the size of Lipschitz constant.
In addition, based on the APL method, we develop a completely problem-parameter free smoothing scheme, 
namely: the uniform smoothing level (USL) method, for solving the aforementioned class of 
structured CP problems with a bilinear saddle point structure~\cite{Nest05-1}. 
We show that this method can find an $\epsilon$-solution of these CP problems 
in at most ${\cal O} (1/\epsilon)$ iterations.

Finally, we demonstrate through our preliminary numerical experiments that these new BL type methods 
can be competitive and even significantly outperform existing first-order 
methods for solving certain classes of CP problems. Observe that each iteration of BL type methods involves the projection onto
$X$ coupled with a few linear constraints, while gradient type methods only require the projection onto $X$. 
As a result, the iteration cost
of BL type methods can be higher than that of gradient type methods, especially
when the projection onto $X$ has explicit solutions. Here
we would like to highlight a few interesting cases in which the application of BL type methods
would be preferred:
(i) the major iteration cost does not exist in the projection onto $X$,
but the computation of first-order information (e.g., involving eigenvalue decomposition or the
solutions of another optimization problem); and
(ii) the projection onto $X$ is 
as expensive as the projection onto $X$ coupled with a few linear constraints, e.g., 
$X$ is a general polyhedron. 
In particular, we show that the APL and USL methods, when applied to solving certain important classes of
semidefine programming (SDP) and stochastic programming (SP) 
problems, can significantly outperform gradient type algorithms,
as well as some existing BL type methods.
The problems we tested consist of instances with up to $77,213$ decision variables.

The paper is organized as follows. In Section~\ref{sec-ABL},
we provide a brief review of the BL method and present
the ABL method for black-box CP problems. We then study a restricted memory version of
the ABL method, namely the APL method in Section \ref{sec-APL}.
In Section~\ref{sec-app}, we investigate how to generalize 
the APL method for solving certain composite and structured nonsmooth CP problems. 
Section~\ref{sec-num} is dedicated to the numerical experiments conducted on
certain classes of SDP and SP problems. 
Finally, some concluding remarks are made in Section \ref{c_remarks}.

\setcounter{equation}{0}
\section{The accelerated bundle-level method} \label{sec-ABL}
We present a new BL type method, namely: the accelerated 
bundle-level (ABL) method, which can uniformly achieve the optimal rate of convergence for
smooth, weakly smooth and nonsmooth CP problems. More specifically, we 
provide a brief review of the BL method for
nonsmooth minimization in Section~\ref{sec-BL}, and then present the ABL method 
and discuss its
main convergence properties in Section~\ref{sec-ABL-alg}.
Section~\ref{sec-ABL-analysis} is devoted to the proof of a major convergence result
used in Section~\ref{sec-ABL-alg}. Throughout this section, we assume that the Euclidean space
$\bbr^n$ is equipped with the standard Euclidean norm $\|\cdot\|$ associated with
the inner product $\langle \cdot, \cdot \rangle$.

\subsection{Review of the bundle-level method} \label{sec-BL}
Given a sequence of search points $x_1, x_2, \ldots, x_k \in X$,
an important construct, namely, the cutting plane model, of the objective
function $f$ of problem \eqnok{cp} is given by
\beq \label{cutting_plane}
m_k(x) := \max \left\{ h(x_i, x): 1 \le i \le k \right\},
\eeq
where
\beq \label{def_Linear_Model}
h(z, x) := f(z) + \langle f'(z), x - z \rangle.
\eeq
In the simplest cutting plane method \cite{CheGol59,Kelley60}, we approximate $f$
by $m_k$ and update the search points according to
\beq \label{sub_cut}
x_{k+1} \in \Argmin_{x \in X} m_k(x).
\eeq
However, this scheme converges slowly, both theoretically and 
practically \cite{nemyud:83,Nest04}. 
A significant progress \cite{Kiw83-1,Kiw90-1,Lem75} was made
under the name of bundle methods (see, e.g., \cite{HeRe00,OliSagSch11-1} for some important
applications of these methods). In these methods, a prox-term 
is introduced into the objective function of \eqnok{sub_cut}
and the search points are updated by
\[
x_{k+1} \in \Argmin_{x \in X} \left\{
m_k(x) + \frac{r_k}{2} \|x - x_k^+\|^2 \right\}.
\]
Here, the current prox-center $x_k^+$ is a certain point from
$\{x_1, \ldots, x_k\}$ and $r_k$ denotes the current penalty parameter.
Moreover, the prox-center for the next iterate, i.e., $x_{k+1}^+$,
will be set to $x_{k+1}$ if $f(x_{k+1})$ is sufficiently smaller
than $f(x_k)$. Otherwise, $x_{k+1}^+$ will be the same as 
$x_k^+$. The penalty $r_k$ reduces the influence of the model $m_k$'s
inaccuracy and hence the instability of the algorithm.
Note, however, that the determination of $r_k$
usually requires certain on-line adjustments or line-search.
In the closely related trust-region technique~\cite{Rusz06,linWri03-1},
the prox-term is put into the constraints of the subproblem instead of its
objective function and the search points are then updated according to 
\[
x_{k+1} \in \Argmin_{x \in X} \left\{
m_k(x):  \|x - x_k^+\|^2 \le R_k \right\}.
\]  
This approach also encounters similar difficulties
for determining the size of $R_k$.

In an important work~\cite{LNN}, Lemar\'{e}chal et al.
introduced the idea of incorporating level sets into the bundle method. The basic scheme 
of their bundle-level (BL) methods consists of:
\begin{itemize}
\item [a)] Update $\overline{f}^k$ to be the best objective value found so far and compute 
a lower bound on $f^*$ by
$
\underline{f}_k = \min_{x\in X} m_k(x);
$
\item [b)] Set $l_k = \lambda \overline{f}^k + (1 - \lambda) \underline{f}_k$ for some $\lambda \in (0,1)$;
\item [c)] Set $x_{k+1} = \argmin_{x \in X} \left\{
\|x - x_k\|^2: m_k(x) \le l_k \right\}.
$
\end{itemize}
Observe that step c) ensures that the new search point $x_{k+1}$ 
falls within the level set $\{x \in X: m_k(x) \le l_k\}$, while being as 
close as possible to $x_k$. We refer to $x_k$ as the {\sl prox-center}, since it controls
the proximity between $x_{k+1}$ and the aforementioned level set. 
It is shown in \cite{LNN} that, if $f$ is a
general nonsmooth convex function (i.e., $\rho = 0$ in \eqnok{smoothness}), then
 the above scheme can find {\sl an $\epsilon$-solution
of \eqnok{cp}} in at most 
\beq \label{old_bl}
{\cal O}\left( C(\lambda) \frac{ M^2 D_X^2} {\epsilon^2}\right)
\eeq
iterations,  
where $C(\lambda)$ is a constant depending on $\lambda$ and
\beq \label{def_DX}
D_X := \max_{x, y \in X} \|x - y\|.
\eeq
In view of \cite{nemyud:83}, the above complexity bound in \eqnok{old_bl}
is unimprovable for nonsmooth convex optimization.
Moreover, it turns out that the level sets give a stable description about
the objective function $f$ and, as a consequence, very good practical performance
has been observed for the BL methods, e.g.,~\cite{LNN,BenNem00,lns11}.

\subsection{The ABL algorithm and its main convergence properties} \label{sec-ABL-alg}

Based on the bundle-level method, our goal in this subsection is to 
present a new bundle type method, namely the ABL method, which can achieve the
optimal complexity for solving any CP problems satisfying \eqnok{smoothness}.

We introduce the following two key improvements into the classical BL
methods. Firstly, rather than using a single sequence $\{x_k\}$,
we employ three related sequences, i.e., $\{x_k^l\}$, $\{x_k^u\}$ and $\{x_k\}$, 
to build the cutting-plane models $m_k(x)$ (and hence the lower bound $\underline {f}_{k}$), 
compute the upper bounds $\overline {f}_{k}$, and control the proximity, respectively. 
Moreover, the relations among these sequences are defined carefully. In particular,
we define $x_k^l = (1-\alpha_k) x_{k-1}^u + \alpha_k x_{k-1}$ and
$x_k^u = (1-\alpha_k) x_{k-1}^u + \alpha_k x_k$ for a certain
$\alpha_k \in (0,1]$. This type of multi-step scheme originated from the well-known
Nesterov's accelerated gradient method for solving smooth CP problems~\cite{Nest83-1}.
Secondly, we group the iterations performed by the ABL method into different phases, and
in each phase, the gap between the lower and upper bounds on $f^*$ will be reduced
by a certain constant factor. It is worth noting that,
although the convergence analysis of the BL method also relies on the concept of phases (see, e.g., \cite{BenNem00,BenNem05-1}),
the description of this method usually does not involve phases.
However, we need to use phases explicitly in the ABL method in order
to define $\{\alpha_k\}$ in an optimal way to achieve the best possible 
complexity bounds for solving problem \eqnok{cp}.

We start by describing the ABL gap reduction procedure, which, 
for a given search point $p$ and lower bound $\lb$ on $f^*$, computes a new search point $p^+$ and  
updated lower bound $\lb^+$ satisfying $f(p^+) - \lb^+ \le \lambda \, [f(p) - \lb]$
for some $\lambda \in (0,1)$. 

\vgap

\noindent {\bf The ABL gap reduction procedure}: $(p^+, \lb^+) = \ABLGAP(p, \lb, \textcolor{blue}{\lambda})$
\begin{itemize}
\item [0)] Set $x^u_0=p$, $\overline{f}_0 = f(x^u_0)$, and $\underline{f}_0 = \lb$.
Also let $x_0 \in X$ and the cutting plane $m_0(x)$ be arbitrarily chosen, 
say $x_0 = x^u_0$ and $m_0(x) = h(x_0, x)$. Let $k=1$.
\item [1)]   {\sl Update lower bound:} set
$x^l_k = (1 - \alpha_k) \, x^u_{k-1} + \alpha_k \, x_{k-1}$, $h(x^l_k, x) = f(x^l_k) + \langle f'(x^l_k), x - x^l_k \rangle$,
$m_k(x) = \max\left\{m_{k-1}(x), h(x^l_k, x)\right\}$,
\beq \label{def_ABL_LB}
h_{k}^* = \min_{x\in X} m_{k}(x) \ \ \mbox{and} \ \ \
\underline{f}_{k} = \max\{\underline{f}_{k-1}, h_{k}^*\};
\eeq
\item [2)]   {\sl Update prox-center:} set $l_k = \lambda \underline{f}_k + (1-\lambda) \overline{f}_{k-1}$ 
and
\beq \label{def_ABL_xt}
x_{k} = \argmin \left\{
\|x - x_{k-1}\|^2: m_k(x) \le l_k, x \in X \right\};
\eeq
\item [3)] {\sl Update upper bound:} 
set
\textcolor{blue}{$\overline{f}_k = \min\{\overline{f}_{k-1}, f(\alpha_k x_k + (1 - \alpha_k) x^u_{k-1})\}$},
and choose $x^u_k \in X$ such that $f(x^u_k) = \overline{f}_k$;
\item [4)] If $\overline{f}_k - \underline{f}_k \le
\lambda (\overline{f}_0 - \underline{f}_0)$,
{\bf terminate} the procedure with $p^+ = x^u_k$ and $\lb^+= \underline{f}_k$;
\item [5)] Set $k = k+1$ and go to Step 1.
\end{itemize}

\vgap

We now add a few remarks about the above gap reduction procedure $\ABLGAP$.
Firstly, we say that an {\sl iteration of procedure $\ABLGAP$} occurs
whenever $k$ increases by $1$. 
Observe that, if $\alpha_k = 1$ for all $k \ge 1$,
then an iteration of procedure~$\ABLGAP$ will be exactly the same as that of the BL method.
In fact, in this case procedure~$\ABLGAP$ will reduce to one phase of the BL method as described in~\cite{BenNem05-1,BenNem00}.
Secondly, with more general selections of $\{\alpha_k\}$, the iteration cost of procedure~$\ABLGAP$ is
still about the same as that of the BL method. More specifically, 
each iteration of procedure~$\ABLGAP$ involves the solution of two subproblems, i.e., 
\eqnok{def_ABL_LB} and \eqnok{def_ABL_xt}, and the computation of $f(x^l_k), f'(x^l_k)$ and $f(\alpha_k x_k + (1- \alpha_k)x^u_{k-1})$,
while the BL method requires the solution of two similar subproblems 
and the computation of $f(x_k)$ and $f'(x_k)$. 
Thirdly, it can be easily seen that $\underline{f}_k$ and $\overline{f}_k$, $k \ge 1$, respectively,
computed by procedure $\ABLGAP$ are lower and upper bounds on $f^*$. Indeed, by the definition of $m_k(x)$,
\eqnok{def_Linear_Model} and the convexity of $f$, we have
\beq \label{better_appr}
m_1(x) \le  m_2(x) \le \ldots m_k(x) \le f(x), \ \ \
\forall \, x\in X,
\eeq
which, in view of \eqnok{def_ABL_LB}, then implies that
$
\underline{f}_{1} \le \underline{f}_2 \le \ldots \underline{f}_k \le f^*.
$
Moreover, it follows from the definition of $\overline{f}_k$ that
$
\overline{f}_{1} \ge \overline{f}_{2} \ge \ldots \ge \overline{f}_{k} \ge f^*.
$
Hence, denoting 
\beq \label{def_Delta}
\Delta_k := \overline{f}_k - \underline{f}_k, \ \ k \ge 0,
\eeq
we have
\beq \label{non-increasing-gap0}
\Delta_0 \ge \Delta_{1} \ge \Delta_{2} \ge \ldots \ge \Delta_{k}
\ge 0.
\eeq

By showing how $\Delta_k$ in \eqnok{def_Delta} decreases with respect to $k$,
we establish in Theorem~\ref{ABL_sum} some important convergence properties
of procedure~$\ABLGAP$. The proof of this result is more involved and hence provided 
separately in Section~\ref{sec-ABL-analysis}.

\begin{theorem} \label{ABL_sum}
Let $\lambda \in (0,1)$ and $\alpha_k \in (0,1]$, $k = 1, 2,\ldots$, be given.
Also let $\Delta_k = \overline{f}_k - \underline{f}_k$ denote the optimality gap obtained
at the $k$-th iteration of procedure $\ABLGAP$ before it terminates. Then for any $k = 1, 2, \ldots$, we have
\beq \label{Recursion_ABL}
\Delta_k \le \gamma_k(\lambda) \left[\textcolor{blue}{(1-\lambda \alpha_1)} \Delta_0 + \frac{M D_X^{1+\rho}}{1+\rho}  \|\Gamma_{k}(\lambda, \rho)\|_{\frac{2}{1-\rho}}\right],
\eeq
where  $D_X$ is defined in \eqnok{def_DX}, $\|\cdot\|_p$ is the $l_p$ norm, 
\beq \label{def_ABL_gamma}
\gamma_k(\lambda) := 
\left\{ 
\begin{array}{ll}
1 & k = 1,\\
(1- \lambda \alpha_k) \, \gamma_{k-1}(\lambda)  & k \ge 2,
\end{array}
\right. \ \ \mbox{and}
\eeq
\beq \label{def_ABL_gamma1}
\Gamma_k(\lambda, \rho) := \left\{\gamma_1(\lambda)^{-1} \alpha_1^{1+\rho}, \gamma_2(\lambda)^{-1} \alpha_2^{1+\rho}, \ldots, \gamma_k(\lambda)^{-1} \alpha_k^{1+\rho}  \right\}.
\eeq
In particular, if $\lambda$ and $\alpha_k \in (0,1]$, $k=1, 2, \ldots$, are chosen such that
for some $c_1, c_2 > 0$,
\beq \label{ABL_policy}
\gamma_k(\lambda) \le c_1k^{-2} \ \ \
\mbox{and} \ \ \
\gamma_k(\lambda) \|\Gamma_k(\lambda,\rho)\|_\frac{2}{1-\rho} \le c_2 k^{-\frac{1+3\rho}{2}},
\eeq
then the number of iterations performed
by procedure~$\ABLGAP$ can be bounded by
\beq \label{ABL_GAP_BND}
K_{ABL}(\Delta_0) := \left \lceil 
\sqrt{\frac{2 c_1 \textcolor{blue}{(1-\lambda \alpha_1)}}{\lambda}} + \left(
\frac{2 c_2 M D_X^{1+\rho}}{(1+\rho)\Delta_0}
\right)^\frac{2}{1+3\rho}
\right \rceil.
\eeq
\end{theorem}

\vgap

\textcolor{blue}{Observe that, if $\alpha_k = 1$
for all $k \ge 1$, then as mentioned before,
procedure~$\ABLGAP$ reduces to a single phase (or segment) of the BL method
and hence its termination follows by slightly modifying the standard analysis of the BL algorithm.
However, such a selection of $\{\alpha_k\}$ does not satisfy
the conditions stated in~\eqnok{ABL_policy} and thus cannot guarantee the
termination of procedure~$\ABLGAP$ in at most $K_{ABL}(\Delta_0)$
iterations. 
Below we discuss a few possible selections of $\{\alpha_k\}$ that 
satisfy \eqnok{ABL_policy}, in order to obtain the bound in \eqnok{ABL_GAP_BND}. It should be pointed out
that none of these selections rely on any problem parameters, such as $M$, $\rho$ and $D_X$. }

\begin{proposition} \label{ABL-step}
Let $\gamma_k(\lambda)$ and $\Gamma_k(\lambda, \rho)$, respectively, be defined in \eqnok{def_ABL_gamma}
and \eqnok{def_ABL_gamma1} for some $\lambda \in (0,1)$ and $\rho \in [0,1]$.
\begin{itemize}
\item [a)] If \textcolor{blue}{$\lambda \in (2/3, 1]$ and $\alpha_k = 2/[\lambda (k+2)]$}, $k = 1, 2, \ldots$,
then $\alpha_k \in (0,1)$ and relation \eqnok{ABL_policy} holds with 
\[
c_1 = 6 \ \ \mbox{and} \ \ \ c_2 = \textcolor{blue}{\frac{2^{3-\rho}}{3^\frac{1-\rho}{2} } \lambda^{-(1+\rho)}}.
\]
\item [b)] If $\alpha_k$, $k \ge 1$, are recursively
defined by
\beq \label{stepsize2-abl}
\alpha_1 = \gamma_1 = 1, \ \
\gamma_k = \alpha_k^2 = (1 - \lambda \alpha_k) \gamma_{k-1},
\ \ \forall \, k \ge 2,
\eeq
then we have $\alpha_k \in (0,1]$ for any $k \ge 1$. Moreover,
condition \eqnok{ABL_policy} is satisfied with
\[
c_1 = 4 \lambda^{-2}\ \ \mbox{and} \ \ \ c_2 = \frac{4} {3^{\frac{1-\rho}{2}}}
 \lambda^{-\frac{1+3\rho}{2}}.
\]
\end{itemize}
\end{proposition}

\begin{proof}
Denoting $\gamma_k \equiv \gamma_k(\lambda)$
and $\Gamma_k \equiv \Gamma_k(\lambda, \rho)$,
we first show part a). 
Note that by \eqnok{def_ABL_gamma}, the selection of $\{\alpha_k\}$ 
and the fact that
$\rho \in [0,1]$, we have
\[
\gamma_k = \textcolor{blue}{\frac{6}{(k+1)(k+2)} }\ \ 
\mbox{and} \ \
\gamma_k^{-1} \alpha_k^{1+\rho}  \le \frac{2^{1+\rho}}{6 \,  \lambda^{1+\rho}} (k+2)^{1-\rho}.
\]
Using these relations and the simple observation 
$\sum_{i=1}^k i^2 = k(k+1)(2k+1)/6 
\le k(k+1)^2/3,
$
we conclude that
\beqas
\gamma_k \|\Gamma_{k}\|_\frac{2}{1-\rho}  &\le&  \gamma_k \frac{2^{1+\rho} }{6 \,  \lambda^{1+\rho}}  \left[\sum_{i=1}^k 
(i+2)^2 \right]^\frac{1-\rho}{2} 
\le \gamma_k \frac{2^{1+\rho} }{6 \,  \lambda^{1+\rho}}  \left[ \frac{(k+2)(k+3)^2}{3}\right]^\frac{1-\rho}{2} \\
&=& \frac{2^{1+\rho}}{3^{\frac{1-\rho}{2}} \lambda^{1+\rho}}
 \frac{(k+3)^{1-\rho}}{(k+1) (k+2)^\frac{1+\rho}{2}}
 \le \frac{2^{1+\rho} 4^{1-\rho}}{3^\frac{1-\rho}{2} \lambda^{1+\rho}} k^{-\frac{1+3\rho}{2}}
 = \frac{2^{3-\rho}}{3^\frac{1-\rho}{2} \lambda^{1+\rho}} k^{-\frac{1+3\rho}{2}},
\eeqas
\textcolor{blue}{where the last inequality follows from the facts $k+3 \le 4k$ and $k+2 \ge k+1 \ge k$
for any $k \ge 1$.}

We now show that part b) holds.
Note that by \eqnok{stepsize2-abl}, we have
\beq \label{stepsize2_alpha}
\alpha_k = \frac{1}{2} \left(
-\lambda \gamma_{k-1} + \sqrt{(\lambda \gamma_{k-1})^2 + 4 \gamma_{k-1} } \,
\right), \ \ \ k \ge 2,
\eeq
which clearly implies that $\alpha_k > 0$, $k \ge 2$. 
We now show that $\alpha_k \le 1$ and $\gamma_k \le 1$ by induction.
Indeed, if $\gamma_{k-1} \le 1$, then, by  
\eqnok{stepsize2_alpha}, we have
\beqas
\alpha_k &\le&
\frac{1}{2} \left(
-\lambda \gamma_{k-1} + \sqrt{(\lambda \gamma_{k-1})^2 + 4 \gamma_{k-1} 
+ 4 [1 - (1-\lambda) \gamma_{k-1}]} \,
\right)\\
&=&
 \frac{1}{2} \left( 
-\lambda \gamma_{k-1} + \sqrt{(\lambda\gamma_{k-1})^2 
+ 4 \lambda \gamma_{k-1} + 4  } \right) = 1. 
\eeqas
The previous conclusion, together with the fact that $\alpha_k^2 = \gamma_k$
due to \eqnok{stepsize2-abl}, then also imply that $\gamma_k \le 1$.
Now let us bound $1 / \sqrt{\gamma_k}$ for any $k \ge 2$. 
First observe that, by \eqnok{stepsize2-abl}, we have, for any $k \ge 2$,
\[
\frac{1}{\sqrt{\gamma_k}} - \frac{1}{\sqrt{\gamma_{k-1}}}
= \frac{\sqrt{\gamma_{k-1}} - \sqrt{\gamma_k}}
{\sqrt{\gamma_{k-1} \gamma_k}}
= \frac{\gamma_{k-1} - \gamma_k}{\sqrt{\gamma_{k-1} \gamma_k}
\left( \sqrt{\gamma_{k-1}} + \sqrt{\gamma_k} \right)}
= \frac{ \lambda \alpha_k \gamma_{k-1}}{\gamma_{k-1} \sqrt{\gamma_k}
+ \gamma_k \sqrt{\gamma_{k-1} }}. 
\]
Using the above identity, \eqnok{stepsize2-abl} and the fact that $\gamma_k \le \gamma_{k-1}$ 
due to \eqnok{stepsize2-abl}, we conclude that
\[
\frac{1}{\sqrt{\gamma_k}} - \frac{1}{\sqrt{\gamma_{k-1}}} \ge \frac{\lambda \alpha_k }
{2 \sqrt{\gamma_k}} = \frac{\lambda}{2} \ \ 
\mbox{and} \ \
\frac{1}{\sqrt{\gamma_k}} - \frac{1}{\sqrt{\gamma_{k-1}}} 
\le \frac{\lambda \alpha_k }{\sqrt{\gamma_k}} = \lambda,
\]
which, in view of the fact that $\gamma_1 = 1$, then implies that
$
1+ \lambda (k-1)/2 \le 1/\sqrt{\gamma_k} \le 1+\lambda (k-1).
$
Using the previous inequality and \eqnok{stepsize2-abl}, we conclude that
\[
\gamma_k \le \frac{4}{[2+\lambda (k-1)]^2} \le \frac{4}{\lambda^2 k^2}, \ \ \
\gamma_k^{-1} \alpha_k^{1+\rho} = 
\left(\sqrt{\gamma_k}\right)^{-(1-\rho)} \le [1+\lambda(k-1)]^{1-\rho},
\]
and
\beqas
\gamma_k \|\Gamma_{k}\|_\frac{2}{1-\rho} &\le&   \gamma_k \left[\sum_{i=1}^k 
[1+\lambda(i-1)]^2 \right]^\frac{1-\rho}{2} 
\le \gamma_k  \left( \int_0^{2+\lambda(k-1)} u^2 \, du\right)^\frac{1-\rho}{2} \\
&\le& \frac{4} {3^{\frac{1-\rho}{2}}}  [2+\lambda (k-1)]^{-\frac{1+3\rho}{2}} 
\le \frac{4} {3^{\frac{1-\rho}{2}}} (\lambda k)^{-\frac{1+3\rho}{2}}.
\eeqas
\end{proof}

\vgap

According to the termination criterion in step 4 of procedure~$\ABLGAP$,
each call to this procedure will reduce the gap between a given
upper and lower bound on $f^*$ by a constant factor. In the ABL method
described below, we will iteratively
call procedure~$\ABLGAP$ until a certain accurate solution of problem
\eqnok{cp} is found.

\vgap

\noindent {\bf The ABL method:} 
\begin{itemize}
\item [] {\bf Input:} initial point $p_0 \in X$, tolerance $\epsilon > 0$ \textcolor{blue}{and algorithmic parameter $\lambda \in (0,1)$}.
\item [0)] Set $p_1 \in \Argmin_{x \in X} h(p_0, x)$, $\lb_1 = h(p_0, p_1)$
and $\ub_1 = f(p_1)$. Let $s=1$.
\item [1)] If $\ub_s - \lb_s \le \epsilon$, {\bf terminate};
\item [2)] Set $(p_{s+1}, \lb_{s+1}) = \textcolor{blue}{ \ABLGAP(p_s, \lb_s, \lambda)}$ and $\ub_{s+1} = f(p_{s+1})$;
\item [3)] Set $s = s+1$ and go to step 1.
\end{itemize}

Whenever $s$ increments by $1$, we say that a phase of the ABL method occurs.
Unless explicitly mentioned otherwise, an iteration of
procedure~$\ABLGAP$ is also referred to as an iteration of the ABL method.
The main convergence properties of the above ABL method are summarized as follows.

\begin{theorem} \label{ABL_theorem}
Suppose that $\lambda \in (0,1)$ and $\alpha_k\in (0,1]$, $k = 1, 2, \ldots$,
in procedure~$\ABLGAP$ are chosen such that \eqnok{ABL_policy} holds
for some $c_1, c_2 >0$.
Let $D_X$, $M$ and $\rho$ be given by \eqnok{def_DX} and \eqnok{smoothness}.

\begin{itemize}
\item [a)] The number of phases performed by the ABL method does not exceed
\beq \label{bnd_phase}
S(\epsilon) = \left \lceil \max \left\{0,
\log_{\frac{1}{\lambda}}
 \frac{M D_X^{1+\rho}}{(1+\rho) \epsilon}\right\}
\right \rceil.
\eeq
\item [b)] The total number of iterations
performed by the ABL method can be bounded by
\beq \label{bound_ABL_iter}
\left(1 + \sqrt{\frac{2 c_1}
{\lambda} }\right) \, S(\epsilon)+
\frac{1}{1 - \lambda^\frac{2}{1+3\rho}}
\left(
\frac{2 c_2 M D_X^{1+\rho}}{(1+\rho)\epsilon}
\right)^\frac{2}{1+3\rho} .
\eeq
\end{itemize}
\end{theorem}

\begin{proof}
Denote $\delta_s \equiv \ub_s - \lb_s$, $s \ge 1$. 
Without loss of generality, we assume that $\delta_1 > \epsilon$, 
since otherwise the statements are obviously true.
Note that by the origin of $\ub_s$ and $\lb_s$, we have 
\beq \label{de_delta}
\delta_{s+1} \le \lambda \delta_s, \ \ s \ge 1.
\eeq
Also note that, by \eqnok{smoothness}, \eqnok{def_DX} and the definition of $p_1$
in the ABL method, we have
\beq \label{bound_delta0}
\delta_{1} = f(p_1) - h(p_0, p_1) = f(p_1) - \left[ f(p_0) + \langle f'(p_0),
p_1 - p_0 \rangle \right] \le \frac{M \|p_1 - p_0\|^{1+\rho}}{1+\rho} \le \frac{M D_X^{1+\rho}}{1+\rho}.
\eeq
The previous two observations then clearly imply that the number of phases performed by
the ABL method is bounded by \eqnok{bnd_phase}. 
We now bound the total number of iterations performed by the ABL method. 
Suppose that procedure $\ABLGAP$ has been called
$\bar s$ times for some $1\le \bar s \le S(\epsilon)$.
It then follows from \eqnok{de_delta} that
$\delta_{s} > \epsilon \lambda^{s-\bar s}$, $s=1, \ldots, \bar s$, since
$\delta_{\bar s} > \epsilon$ due to the origin of $\bar s$.
Using this observation, we obtain
\[
\sum_{s=1}^{\bar s} \delta_{s}^{-\frac{2}{1+3\rho}}
 < \sum_{s=1}^{\bar s} \frac{ \lambda^{\frac{2}{1+3\rho}(\bar s  -s)}}{\epsilon^{\frac{2}{1+3\rho}} }
 = \sum_{t=0}^{\bar s-1} \frac{ \lambda^{\frac{2}{1+3\rho}t}}{\epsilon^{\frac{2}{1+3\rho}} }
   \le \frac{1}{(1 - \lambda^\frac{2}{1+3\rho}) \epsilon^\frac{2}{1+3\rho}
 }.
\]
Moreover, by Theorem~\ref{ABL_sum}, the total number of iterations performed by the ABL method is bounded by
\[
\sum_{s=1}^{\bar s} K_{ABL}(\delta_s) \le \left(1 + \sqrt{\frac{2 c_1}{\lambda}}\right) \bar s 
+\sum_{s=1}^{\bar s} \left(
\frac{2 c_2 M D_X^{1+\rho}}{(1+\rho)\delta_s}
\right)^\frac{2}{1+3\rho} \\
\]
Our result then immediately follows by combining the above two inequalities.
\end{proof}

\vgap

We now add a few remarks about Theorem~\ref{ABL_theorem}.
Firstly, 
by setting $\rho = 0$, $\rho =1$ and $\rho \in (0,1)$ in \eqnok{bound_ABL_iter}, respectively, we obtain
the optimal iteration complexity for nonsmooth, smooth and weakly smooth
convex optimization (see \cite{Lan10-3,NemNes85-1,Nest88-1,DeGlNe10-1} for a discussion about the lower complexity bounds
for solving these CP problems). 
Secondly, the ABL method achieves these aforementioned optimal complexity bounds 
without requiring the input of any smoothness information, such as whether the problem is smooth or not,
and the specific values for $\rho$ and $M$ in \eqnok{smoothness}.
To the best of our knowledge, the ABL method seems to be the first
uniformly optimal method for solving smooth, nonsmooth and weakly smooth
CP problems in the literature.
Thirdly, observe that one potential problem for the ABL method is that, as the algorithm proceeds,
the model $m_k(x)$ accumulates cutting planes, 
and the subproblems in procedure~$\ABLGAP$ become more difficult to solve. 
We will address this issue in Section~\ref{sec-APL} by developing
a variant of the ABL method.

\subsection{Convergence analysis of the ABL gap reduction procedure} \label{sec-ABL-analysis}
Our goal in this subsection is to prove Theorem~\ref{ABL_sum}, which describes
some important convergence properties of procedure~$\ABLGAP$.
We will first establish three technical results from which Theorem~\ref{ABL_sum} immediately follows. 

\textcolor{blue}{
Lemma~\ref{ABL-cluster} below shows that the prox-centers $\{x_k\}$ for procedure~$\ABLGAP$ are ``close'' to each other,
in terms of $\sum_{i = 1}^{k} \|x_{i-1} - x_{i}\|^2$. It follows
the standard analysis of the BL method (see, e.g., \cite{LNN,BenNem00}).}

\begin{lemma} \label{ABL-cluster}
Let $m_i(x)$ and $l_i$, $i = 1, \ldots, k$, respectively, be computed in step 1 and step 2 of procedure~$\ABLGAP$
before it terminates. Then the level sets given by
$
{\cal L}_{i} := \left\{ x \in X: m_{i}(x) \le l_i \right\}, i = 1, 2, \ldots,k, 
$
have a point in common. As a consequence, we have
\beq \label{ABL_sum_dist}
\sum_{i = 1}^{k} \|x_{i-1} - x_{i}\|^2 \le D_X^2,
\eeq
where $D_X$ is defined in \eqnok{def_DX}.
\end{lemma}

\begin{proof}
Let $\Delta_k$, $k = 0, 1, \ldots$,  be defined in \eqnok{def_Delta}.
First, in view of \eqnok{non-increasing-gap0} and the termination criterion of procedure~$\ABLGAP$, 
we have
\beq \label{non-increasing-gap}
\Delta_{0} \ge \Delta_{1} \ge \ldots \ge \Delta_{k} > \lambda \Delta_0.
\eeq
Now let $u \in \Argmin_{x \in X} m_k(x)$. 
Observe that, by \eqnok{def_ABL_LB}, \eqnok{better_appr} and \eqnok{non-increasing-gap},
we have, for any $i = 1, 2, \ldots, k$,
\beqas
m_i(u) &\le& m_k(u) = \underline{f}_k = \overline{f}_k - \Delta_k <
\overline{f}_k - \lambda \Delta_0 \le \overline{f}_{i} - \lambda
\Delta_0 \\
&\le& \overline{f}_{i} - \lambda \Delta_{i} = (1-\lambda) \, \overline{f}_i
+ \lambda \, \underline{f}_i \le (1-\lambda)\, \overline{f}_{i-1} + \lambda \, \underline{f}_i = l_{i}.
\eeqas
We have thus shown that $u \in {\cal L}_{i}$ for any $i= 1, 2, \ldots, k$.
Now by \eqnok{def_ABL_xt},
we have
\[
\|x_i- u\|^2 + \|x_{i-1} - x_i\|^2 \le \|x_{i-1}-u\|^2,
i = 1, 2, \ldots, k.
\]
Summing up the above inequalities and using \eqnok{def_DX},
we obtain
\[
\|x_k-u\|^2 + \sum_{i=1}^k \|x_{i-1} - x_i\|^2 \le \|x_{0}-u\|^2
\le D_X^2,
\]
which clearly implies \eqnok{ABL_sum_dist}.
\end{proof}

\vgap

\textcolor{blue}{
The following two technical results will be used in the convergence
analysis for a few accelerated bundle-level type methods,
including ABL, APL and USL, developed in this paper.}

\textcolor{blue}{
\begin{lemma} \label{three_points}
Let $(x_{k-1}, x^u_{k-1}) \in X \times X$ be given at
the $k$-th iteration, $k \ge 1$, of an iterative scheme
and denote $x^l_k = \alpha_k x_{k-1} + (1-\alpha_k) x^u_{k-1}$.
Also let $h(z, \cdot)$ be defined in \eqnok{def_Linear_Model} and suppose that the pair 
of new search points $(x_k, \tilde x^u_k) \in X \times X$ 
satisfy that, for some $l \in \bbr$ and $\alpha_k \in (0,1]$,
\begin{align}
h(x^l_k, x_k) \le l, \label{rel2}\\
\tilde x^u_k = \alpha_t x_k + (1-\alpha_k) x^u_{k-1}. \label{rel3}
\end{align}
Then, 
\beq \label{recursion}
f(\tilde x^u_k) \le (1 - \alpha_k) f(x^u_{k-1})  + \alpha_k l +  
\frac{M}{1+\rho} \|\alpha_k (x_k - x_{k-1})\|^{1+\rho}.
\eeq
\end{lemma}
}
\begin{proof}
It can be 
easily seen from \eqnok{rel3} and the definition of $x_k^l$ that 
\beq \label{temp_rel_abl}
\tilde x^u_k - x^l_k = \alpha_k (x_k - x_{k-1}).
\eeq 
Using this observation, \eqnok{smoothness}, \eqnok{def_Linear_Model}, \eqnok{rel2}, \eqnok{rel3} and
the convexity of $f$, we have
\begin{align}
f(\tilde x^u_k) &\le h(x^l_k,\tilde x^u_k) + \frac{M}{1+\rho} \|\tilde x^u_k - x^l_k\|^{1+\rho} 
& & \textcolor{blue}{\text{(by \eqnok{smoothness} and \eqnok{def_Linear_Model})}}\nn\\
&= (1 - \alpha_k) h(x^l_k, x^u_{k-1}) + \alpha_k h(x^l_k, x_k) + \frac{M}{1+\rho} \|\tilde x^u_k - x^l_k\|^{1+\rho} 
& & \textcolor{blue}{\text{(by \eqnok{rel3})}}\nn\\
&= (1 - \alpha_k) h(x^l_k, x^u_{k-1}) + \alpha_k h(x^l_k, x_k)
+ \frac{M }{1+\rho} \|\alpha_k(x_k - x_{k-1})\|^{1+\rho} \nn
& & \textcolor{blue}{\text{(by \eqnok{temp_rel_abl})}}\\
&\le (1 - \alpha_k) f(x^u_{k-1}) + \alpha_k h(x^l_k, x_k)
+ \frac{M}{1+\rho} \|\alpha_k (x_k - x_{k-1})\|^{1+\rho} 
& & \textcolor{blue}{\text{(by convexity of $f$)}}\nn \\
&\le (1 - \alpha_k) f(x^u_{k-1}) + \alpha_k l
+ \frac{M}{1+\rho} \|\alpha_k(x_k - x_{k-1})\|^{1+\rho}.
& & \textcolor{blue}{\text{(by \eqnok{rel2})}}\nn
\end{align}
\end{proof}

\textcolor{blue}{
\begin{lemma} \label{tech_result_sum} Let $w_k\in (0,1]$, $k = 1, 2, \ldots$, be given. 
Also let us denote
\beq \label{def_Gamma0}
W_k := \left\{
\begin{array}{ll}
1,& k = 1,\\
(1-w_k) \, W_{k-1},& k \ge 2.
\end{array}
\right.
\eeq
Suppose that $W_k > 0$ for all $k \ge 2$ and that the sequence $\{\delta_k\}_{k \ge 0}$ satisfies
\beq \label{general_cond}
\delta_k \le (1 - w_k) \delta_{k-1} + B_k, \ \ \ k = 1, 2, \ldots.
\eeq
Then, we have
$
\delta_k \le W_k (1-w_1) \delta_0 + W_k \sum_{i=1}^k (B_i/W_i).
$
\end{lemma} }

\begin{proof}
Dividing both sides of \eqnok{general_cond} by $W_k$, we obtain
\[
\frac{\delta_1}{W_1} \le \frac{(1-w_1) \delta_{0}}{W_{1}} + \frac{B_1}{W_1}
\]
and
\[
\frac{\delta_k}{W_k} \le \frac{\delta_{k-1}}{W_{k-1}} + \frac{B_k}{W_k}, \ \ \ \forall k \ge 2.
\]
The result then immediately follows by summing up the above inequalities
and rearranging the terms.
\end{proof}

\vgap

We are now ready to provide the proof of Theorem~\ref{ABL_sum}.

\noindent {\bf Proof of Theorem~\ref{ABL_sum}:} 
\textcolor{blue}{We first establish an important recursion for procedure~$\ABLGAP$.
Let $\Delta_k = \overline{f}_k - \underline{f}_k$, $k =1,2,\ldots$, be the optimality gap computed
at the $k$-th iteration of this procedure and denote $\tilde x^u_k \equiv \alpha_k x_k + (1-\alpha_k) x^u_{k-1}$.
By the definitions of $x^u_k$ and $\tilde x^u_k$, we have $f(x^u_k) \le f(\tilde x^u_k)$.
Also by \eqnok{def_ABL_xt} and the definition of $m_k(x)$, we have $h(x_k^l, x_k) \le l_k$. 
 Using these observations and Lemma~\ref{three_points} (with $l=l_k$), we conclude that,
 for any $k \ge 1$,}
\beq \label{basic_it}
f(x^u_k) \le f(\tilde x^u_k) \le (1 - \alpha_k) f(x^u_{k-1}) + \alpha_k l_k
+ \frac{M}{1+\rho} \|\alpha_k(x_k - x_{k-1})\|^{1+\rho}.
\eeq
Subtracting $\underline{f}_{k}$ from both sides of the above inequality,
and observing that $f(x^u_k) - \underline{f}_{k}
= \overline{f}_{k} - \underline{f}_{k}  = \Delta_{k}$ and $f(x^u_{k-1}) = \overline{f}_{k-1}$,  
we obtain
\[
\Delta_k \le  (1 - \alpha_k) \overline{f}_{k-1} - \underline{f}_{k} + \alpha_k l_k   + \frac{M}{1+\rho} \|\alpha_k(x_k - x_{k-1})\|^{1+\rho}, \ \ \forall k \ge 1.
\]
Also note that
\beqas
(1 - \alpha_k) \overline{f}_{k-1} - \underline{f}_{k} + \alpha_k l_k &=& (1 - \alpha_k) \overline{f}_{k-1} - \underline{f}_{k} + \alpha_k \left[
\lambda \, \underline{f}_{k} + (1-\lambda) \, \overline{f}_{k-1}\right] \\
&=& (1-\alpha_k) \overline{f}_{k-1} + \alpha_k (1-\lambda) \, \overline{f}_{k-1}
-(1 - \lambda \,\alpha_k ) \underline{f}_{k} \\
&\le& (1-\alpha_k) \overline{f}_{k-1} + \alpha_k (1-\lambda) \, \overline{f}_{k-1}
-(1 - \lambda \, \alpha_k ) \underline{f}_{k-1} \\
&=& (1 - \lambda \, \alpha_k ) \Delta_{k-1},
\eeqas
where the inequality follows from the fact that $\underline{f}_{k} \ge \underline{f}_{k-1}$.
Combining the above two inequalities, we arrive at
\beq \label{ABL_it}
\Delta_{k} \le (1 - \lambda \, \alpha_k) \Delta_{k-1}
+ \frac{M}{1+\rho} \|\alpha_k(x_{k} - x_{k-1})\|^{1+\rho}, \ \ \forall \ k \ge 1.
\eeq

\textcolor{blue}{Next, let $\gamma_k \equiv \gamma_k(\lambda)$
and $\Gamma_k \equiv \Gamma_k(\lambda, \rho)$ be 
defined in \eqnok{def_ABL_gamma} and \eqnok{def_ABL_gamma1},
respectively. Using \eqnok{ABL_it} and Lemma~\ref{tech_result_sum}
(with $\delta_k = \Delta_k$, $w_k = 1 - \lambda \alpha_k$, $W_k = \gamma_k$ and $B_k = M\|\alpha_k(x_{k} - x_{k-1})\|^{1+\rho}/(1+\rho)$),
we obtain  }
\beqa
\Delta_k &\le& \gamma_k (1 - \lambda \alpha_1) \Delta_0 + \frac{M}{1+\rho} \sum_{i=1}^k \left[\gamma_i^{-1}
\|\alpha_i(x_{i} - x_{i-1})\|^{1+\rho}  \right] \nn \\
&\le& \gamma_k (1 - \lambda \alpha_1) \Delta_0 + \frac{M}{1+\rho}
\|\Gamma_{k}\|_{\frac{2}{1-\rho}}
\left[\sum_{i=1}^{k} \|x_i - x_{i-1}\|^2\right]^\frac{1+\rho}{2},
\eeqa
where the last relation follows from H\"{o}lder's inequality.
The previous conclusion, in view of \eqnok{ABL_sum_dist}, then clearly
implies \eqnok{Recursion_ABL}. 

Now let us denote $K = K_{ABL}(\Delta_0)$.
It then follows from \eqnok{Recursion_ABL}, \eqnok{ABL_policy}
and \eqnok{ABL_GAP_BND} that
\[
\Delta_K \le c_1 (1-\lambda \alpha_1) K^{-2} \Delta_0 + \frac{c_2 M D_X^{1+\rho}}{1+\rho} K^{-\frac{1+3\rho}{2}}
\le \frac{\lambda}{2} \Delta_0 + \frac{\lambda}{2} \Delta_0 = \lambda \Delta_0,
\]
and hence that procedure~$\ABLGAP$ will terminate in at most $K$ iterations.
\endproof

\setcounter{equation}{0}
\section{The accelerated prox-level method} \label{sec-APL}
One critical issue for the ABL method is that,
as the algorithm proceeds, its subproblems, namely problem
\eqnok{def_ABL_LB} and \eqnok{def_ABL_xt}, accumulate constraints and thus
become more and more difficult to solve. 
Our goal in this section is to present a variant of the ABL method, 
namely the accelerated prox-level method (APL),
which can still uniformly achieve the optimal iteration complexity
for solving smooth, weakly smooth and nonsmooth CP
problems, but has significantly reduced iteration cost than that of the ABL method.
In addition, throughout this section, we assume that $\bbr^n \supseteq X$
is equipped with an arbitrary norm $\|\cdot\|$ (not necessarily 
associated with the inner product) and $\|\cdot\|_*$ denotes its conjugate.
We will employ non-Euclidean prox-functions in the APL algorithm 
to make use of the geometry of the feasible set $X$,
similarly to the NERML algorithm in \cite{BenNem05-1,BenNem00}. 


We first need to introduce a new way to construct lower bounds on $f^*$. 
Let $\levelset_{f}(l)$ denote the level set of $f$ given by 
\beq \label{def_level_set}
\levelset_{f}(l) := \left \{ x \in X: f(x) \le l\right\}.
\eeq
Also for some $z \in X$, let $h(z, x)$ be the cutting plane defined in \eqnok{def_Linear_Model}
and denote
\beq \label{def_bar_h}
\bar{h} := \min \left\{ h(z, x): x \in \levelset_{f}(l) \right\}.
\eeq
Then, it is easy to verify that 
\beq \label{def_tmp_bnd}
\min\{l, \bar{h}\} \le f(x), \ \ \ \forall x \in X.
\eeq
Indeed, if $l \le f^*$, then $\levelset_{f}(l) = \emptyset$, $\bar{h} = + \infty$
and $\min\{l,\bar{h}\} = l$. Hence \eqnok{def_tmp_bnd} is obviously true.
Now consider the case $l > f^*$. Clearly, for an arbitrary optimal solution $x^*$ of \eqnok{cp},
we have $x^* \in \levelset_{f}(l)$.
Moreover, by \eqnok{def_Linear_Model}, \eqnok{def_bar_h} and 
the convexity of $f$, we have \textcolor{blue}{$\bar{h} \le h(z,x) \le f(x)$} for 
any $x \in \levelset_{f}(l)$. Hence, $\bar{h} \le f(x^*) = f^*$
and thus \eqnok{def_tmp_bnd} holds.

Note, however, that to solve problem \eqnok{def_bar_h} is usually
as difficult as to solve the original problem~\eqnok{cp}. 
To compute a convenient lower bound of $f^*$,
we replace $\levelset_{f}(l)$ in \eqnok{def_bar_h} with a
convex and compact set $X'$ satisfying
\beq \label{def_localizer}
\levelset_{f}(l) \subseteq X' \subseteq X.
\eeq
The set $X'$ will be referred to as 
a {\sl localizer} of the level set $\levelset_{f}(l)$.
The following result shows the computation of
a lower bound on $f^*$ by solving such a relaxation
of \eqnok{def_bar_h}.

\begin{lemma} \label{lemma_lower_bnd}
Let $X'$ be a localizer of the level set
$\levelset_{f}(l)$ for some $l\in \bbr$ and
$h(z, x)$ be defined in \eqnok{def_Linear_Model}. Denote
\beq \label{def_h_star}
\underline{h} := \min \left\{ h(z, x): x \in X' \right\}.
\eeq
We have
\beq \label{def_lower_bnd}
\min\{l, \underline{h}\} \le f(x), \ \ \ \forall x \in X.
\eeq
\end{lemma}

\begin{proof}
Note that if $X' = \emptyset$ (i.e., \eqnok{def_h_star}
is infeasible), then $\underline{h} = + \infty$. In this case,
we have $\levelset_{f}(l) = \emptyset$ and $f(x) \ge l$ for any $x \in X$.
Now assume that $X' \neq \emptyset$.
By \eqnok{def_bar_h}, \eqnok{def_localizer} and \eqnok{def_h_star},
we have $\underline{h} \le \bar{h}$, which together with \eqnok{def_tmp_bnd},
then clearly imply \eqnok{def_lower_bnd}.
\end{proof}

\vgap

The second new construct that we will employ in the APL method
is the {\sl prox-function} that generalizes 
the Euclidean distance function $\|\cdot\|^2$ used in the
ABL method (see \eqnok{def_ABL_xt}). More specifically,
consider a convex compact set $X \subseteq \bbr^n$.
A function $\w:X\to \bbr$ is called a prox-function of $X$ with modulus
$\sigma_\w$, if it is differentiable and strongly convex with modulus $\sigma_\w$, i.e.,
\[
\langle \nabla \w(x) - \nabla \w(z), x - z \rangle \ge \sigma_\w \|x-z\|^2, \ \ \
\forall x, z \in X.
\]
Moreover, we denote {\sl the size of $X$ with respect to $\w$} by
\beq \label{def_cal_DX}
\cD_{\w,X}^2 := \max_{x, z \in X} \left\{\w(x) - \w(z) - 
\langle \nabla \w(z), x - z \rangle\right\}.
\eeq
Clearly, we have
\beq \label{diam_X}
\|x - z\|^2 \le \frac{2}{\sigma_\w} \cD_{\w,X}^2 =: \Omega_{\w,X}, \ \ \ \forall x, z \in X.
\eeq 

Similarly to Section~\ref{sec-ABL-alg}, we start by describing  
a new gap reduction procedure, denoted by~$\APLGAP$, which,
for a given search point $p$ and a lower bound $\lb$ on $f^*$, computes a new search point $p^+$ and a new 
lower bound $\lb^+$ satisfying $f(p^+) - \lb^+ \le q \, [f(p) - \lb]$
for some $q \in (0,1)$. \textcolor{blue}{Note that the value of $q$ will depend on the two algorithmic 
input parameters: $\beta, \theta \in (0,1)$}.

\vgap

\noindent{\bf The APL gap reduction procedure:  $(p^+, \lb^+) = \APLGAP(p, \lb, \textcolor{blue}{\beta, \theta})$} 
\begin{itemize}
\item [0)] Set $x^u_0 = p$, $\overline{f}_0 = f(x^u_0)$, $\underline{f}_0 = \lb$ and
$l = \beta \underline{f}_0 + (1 - \beta) \overline{f}_0$. Also let $x_0 \in X$ and the initial localizer $X'_0$
be arbitrarily chosen, say $x_0 = p$ and $X'_0 = X$.
Set the prox-function $d_\w(x) = \w(x) - [\w(x_0) + \langle \w'(x_0), x - x_0\rangle ]$.
Also let $k=1$. 
\item [1)] {\sl Update lower bound:} set $x^l_k = (1-\alpha_k) x^u_{k-1}+\alpha_k x_{k-1}$,  
$h(x^l_k, x) = f(x^l_k) + \langle f'(x^l_k), x - x^l_k \rangle$,
\beq \label{def_lb}
\underline{h}_k := \min_{x \in X'_{k-1}} \left\{ h(x^l_k, x) \right\}, \ \ \mbox{and} \ \ \
\underline{f}_k := \max\left\{\underline{f}_{k-1}, \min\{l, \underline{h}_k\}\right\}. 
\eeq
If $\underline{f}_k \ge l - \theta (l - \underline{f}_0)$,
then {\bf terminate} the procedure with $p^+ = x^u_{k-1}$ and $\lb^+ = \underline{f}_k$;
\item [2)] {\sl Update prox-center:} set 
\beq \label{prox_step}
x_k := \argmin_{x \in X'_{k-1}} \left\{ d_\w(x): h(x^l_k, x) \le l \right\};
\eeq
\item [3)] {\sl Update upper bound:} set \textcolor{blue}{$\bar f_k = \min \{
\bar{f}_{k-1}, f(\alpha_k x_k + (1-\alpha_k) x^u_{k-1}) \}$}, and choose $x^u_k$
such that $f(x^u_k) = \bar f_k$.
If $\overline{f}_k \le l + \theta (\overline{f}_0 - l)$, then 
{\bf terminate} the procedure with $p^+ = x^u_{k}$ and $\lb^+ = \underline{f}_k$;
\item [4)] {\sl Update localizer:} choose an arbitrary $X'_k$ such that
$
\underline{X}_k \subseteq X'_k \subseteq \overline{X}_k,
$
where
\beq \label{bound_X}
\underline{X}_k := \left\{ x \in X'_{k-1}: h(x^l_k, x)
\le l \right\} \ \ \mbox{and} \ \
\overline{X}_k := \left\{ x \in X: \langle \nabla d_\w(x_{k}), 
x-x_k \rangle \ge 0 \right\}; 
\eeq
\item [6)] Set $k = k+1$ and go to step 1.
\end{itemize} 

We now add a few comments about procedure~$\APLGAP$
described above. Firstly, note that the level $l$ used 
in \eqnok{prox_step} is fixed throughout the procedure. 
This is different from the ABL gap reduction procedure 
where the level $l_k$ used in \eqnok{def_ABL_xt}
changes at each iteration. Moreover, instead of having one parameter 
$\lambda$ as in the ABL method, we need to use two parameters (i.e., $\beta$ 
and $\theta$) whose values are fixed a priori, say, $\beta = \theta = 0.5$.

Secondly, procedure~$\APLGAP$ can be terminated in either step 1 or 3.
If it terminates in step~1, then we say that significant progress has been made on 
the lower bound $\underline{f}_k$. Otherwise, if it terminates in step~3,
then significant progress has been made on the upper bound $\overline{f}_k$. 

Thirdly, observe that in step 4 of procedure~$\APLGAP$, 
we can choose any set $X'_k$ satisfying 
$\underline{X}_k \subseteq X'_k \subseteq \overline{X}_k$
(the simplest way is to set $X'_k = \underline{X}_k$ or 
$X'_k = \overline{X}_k$). While the number of constraints 
in $\underline{X}_k$ increases with $k$, the set
$\overline{X}_k$ has only one more constraint than $X$.
By choosing $X'_k$ between these two extremes, we can control the number of
constraints in subproblems \eqnok{def_lb} and \eqnok{prox_step}.
Hence, the iteration cost of procedure~$\APLGAP$ can be considerably smaller than that of procedure
$\ABLGAP$.  

We summarize below a few more observations regarding the execution
of procedure~$\APLGAP$.

\begin{lemma} \label{obs_gap}
The following statements hold for procedure $\APLGAP$.
\begin{itemize}
\item [a)] $\{X'_k\}_{k\ge 0}$ is a sequence of localizers 
of the level set $\levelset_f(l)$;
\item [b)] $\underline{f}_0 \le \underline{f}_1 \le \ldots \le \underline{f}_k
\le f^*$ and $\overline{f}_0 \ge \overline{f}_1 \ge \ldots \ge \overline{f}_k
\ge f^*$ for any $k \ge 1$;
\item [c)] Problem \eqnok{prox_step} is always feasible unless the procedure terminates;
\item [d)] $\emptyset \neq \underline{X}_k \subseteq \overline{X}_k$ for any 
$k \ge 1$ and hence Step 4 is always feasible unless the procedure terminates;
\item [e)] Whenever the procedure terminates, we have
$f(p^+) - \lb^+ \le q \, [f(p) - \lb],$
where
\beq \label{def_q}
q \equiv q(\beta, \theta) := 1 - (1 -\theta) \min\{\beta, 1-\beta\}.
\eeq
\end{itemize}
\end{lemma}

\begin{proof}
We first show part a). Firstly, noting that $\levelset_f(l) \subseteq X'_0$, we can show that
$\levelset_f(l) \subseteq X'_k$, $k \ge 1$, by using induction.
Suppose that $X'_{k-1}$ is a localizer of the level set $\levelset_f(l)$.
Then, for any $x \in \levelset_f(l)$, we have $x \in X'_{k-1}$. Moreover,
by the definition of $h$, we have $h(x^l_k,x) \le f(x) \le l$
for any $x \in \levelset_f(l)$. Using these two observations  
and the definition of $\underline{X}_k$ in \eqnok{bound_X}, 
we have $\levelset_f(l) \subseteq \underline{X}_k$, which, in view of
the fact that $\underline{X}_k \subseteq X'_k$, implies $\levelset_f(l) 
\subseteq X'_k$, i.e., $X'_k$ is a localizer of $\levelset_f(l)$. 

We now show part b). The first relation follows from Lemma \ref{lemma_lower_bnd},  
\eqnok{def_lb}, and the fact that $X'_k$, $k \ge 0$, are localizers of $\levelset_f(l)$
due to part a). The second relation of part b) follows immediately from the definition of 
$\overline{f}_k$, $k \ge 0$. 

To show part c), suppose that problem \eqnok{prox_step} is infeasible. Then,
by \textcolor{blue}{the definition of $\underline{h}_k$} in \eqnok{def_lb}, we have $\underline{h}_k >l$, 
which implies \textcolor{blue}{$\underline{f}_k \ge l$}, which in turn implies that
the procedure should have terminated in step 1 at iteration $k$.

To show part d), note that by part c), 
the set $\underline{X}_k$ is nonempty. Moreover, by the optimality condition
of \eqnok{prox_step} and the definition of $\underline{X}_k$ in \eqnok{bound_X}, we have
$\langle \nabla \w(x_k), x-x_k \rangle \ge 0$ for any $x \in \underline{X}_k$,
which then implies that $\underline{X}_k \subseteq \overline{X}_k$.

We now provide the proof of part e). 
Suppose first that the procedure terminates in step 1 of the $k$-th iteration. We must have
$\underline{f}_k \ge l - \theta (l - \underline{f}_0)$. By using this condition, and the facts that 
$f(p^+) \le \overline{f}_0$ (see part b) and 
$l = \beta \underline{f}_0 + (1 - \beta) \overline{f}_0$, we obtain
\beq \label{de_delta1}
f(p^+) - \lb^+ =f(p^+)-\underline{f}_k
\le \overline{f}_0 - [l- \theta (l - \underline{f}_0)]
= [1 - (1-\beta) (1-\theta)] (\overline{f}_0  - \underline{f}_0).
\eeq
Now suppose that the procedure terminates in step 3 of the $k$-th iteration. We must have
$\overline{f}_k \le l + \theta (\overline{f}_0 - l)$.
By using this condition, and the facts that $\lb^+ \ge \underline{f}_0$
(see Lemma \ref{obs_gap}.b) and $l = \beta \underline{f}_0 
+ (1 - \beta) \overline{f}_0$,
we have
\[
f(p^+) - \lb^+ = \overline{f}_k - \lb^+
\le l + \theta (\overline{f}_0 - l) - \underline{f}_0 
= [1 - (1-\theta) \beta] (\overline{f}_0 - \underline{f}_0 ). 
\]
Part e) then follows by combining the above two relations.
\end{proof}
 
\vgap

By showing how the gap between the upper bound (i.e., $f(x^u_k)$) and the
level $l$ decreases with respect to $k$,
we establish in Theorem~\ref{APL_sum} some important convergence
properties of procedure~$\APLGAP$.

\begin{theorem} \label{APL_sum}
Let $\alpha_k \in (0,1]$, $k = 1, 2, \ldots$,
be given. Also let $(x^l_k, x_k, x^u_k) \in X \times X \times X$, $k \ge 1$, be the
search points, $l$ be the level and $d_\w(\cdot)$ be the prox-function
in procedure~$\APLGAP$. Then, we have
\beq \label{Recursion_APL}
f(x_k^u) -l \le  (1-\alpha_1) \gamma_k(1) \, [f(x_0^u) -l] +  \frac{M}{1+\rho}\left[\frac{2 d_\w(x_k)}{\sigma_\w}\right]^\frac{1+\rho}{2} 
\gamma_k(1) \, \|\Gamma_{k}(1, \rho)\|_{\frac{2}{1-\rho}}
\eeq
for any
$k \ge 1$, where $\|\cdot\|_p$ denotes the $l_p$ norm, $\gamma_k(\cdot)$ and $\Gamma_k(\cdot, \cdot)$, 
respectively, are defined in \eqnok{def_ABL_gamma} and \eqnok{def_ABL_gamma1}.
In particular, if $\alpha_k \in (0,1]$, $k=1, 2, \ldots$, are chosen such that
for some $c > 0$,
\beq \label{APL_policy}
\alpha_1 = 1 \ \ \
\mbox{and} \ \ \
\gamma_k(1) \, \|\Gamma_k(1,\rho)\|_\frac{2}{1-\rho} \le c \, k^{-\frac{1+3\rho}{2}},
\eeq
then the number of iterations performed
by procedure~$\APLGAP$ can be bounded by
\beq \label{APL_GAP_BND}
K_{APL}(\Delta_0) := \left \lceil 
\left(
\frac{c \, M \Omega_{\w,X}^\frac{1+\rho}{2}}{\beta \theta (1+\rho)\Delta_0}
\right)^\frac{2}{1+3\rho}
\right \rceil,
\eeq
where $\Delta_0 = \overline{f}_0 - \underline{f}_0$
and $\Omega_{\w,X}$ is defined in \eqnok{diam_X}.
\end{theorem}

\begin{proof}
We first show that the prox-centers $\{x_k\}$ in procedure $\APLGAP$ are ``close''
to each other in terms of $\sum_{i=1}^k \|x_{i-1} - x_i\|^2$.
\textcolor{blue}{This result slightly extends Lemma~\ref{ABL-cluster} for procedure~$\ABLGAP$.}
Observe that, the function $d_\w(x)$ is strongly
convex with modulus \textcolor{blue}{$\sigma_\w$}, $x_0 = \arg \min_{x \in X} d_\w(x)$ and $d_\w(x_0) =0$.
Hence, we have,
\beq \label{strong_ws}
\frac{\sigma_\w}{2} \|x_1-x_0\|^2 \le d_\w(x_1) - d_\w(x_0) = d_\w(x_1).
\eeq
Moreover, by \eqnok{bound_X}, we have
$\langle \textcolor{blue}{\nabla d_\w(x_{k})}, x-x_k \rangle \ge 0$ for any $x \in \overline{X}_k$,
which, together with the fact that $X'_k \subseteq \overline{X}_k$, then imply
that $\langle \nabla d_\w(x_{k}), x-x_k \rangle \ge 0$ for any $x \in X'_k$. Using this observation,
the fact that $x_{k+1} \in X'_k$ due to \eqnok{prox_step}, and
the strong convexity of $d_\w$, we have
\[
\frac{\sigma_\w}{2} \|x_{k+1} - x_{k}\|^2 \le 
d_\w(x_{k+1}) - d_\w(x_{k}) - \langle \nabla d_\w(x_{k}), x_{k+1} - x_{k} \rangle
\le d_\w(x_{k+1}) - d_\w(x_{k}), \ \ \forall k \ge 1.
\]
Summing up the above inequalities with \eqnok{strong_ws}, we arrive at 
\beq \label{cluster_APL}
\frac{\sigma_\w}{2} \sum_{i=1}^k \|x_i - x_{i-1}\|^2 \le d_\w(x_k).
\eeq

\textcolor{blue}{Next, we establish a recursion for procedure~$\APLGAP$.
Let us denote $\tilde x_k^u \equiv \alpha_k x_k + (1 - \alpha_k) x_{k-1}^u$,
$\gamma_k \equiv \gamma_k(1)$ and $\Gamma_k \equiv \Gamma_k(1,\rho)$.
By the definitions of $x^u_k$ and $\tilde x^u_k$, we have $f(x_k^u) \le f(\tilde x_k^u)$.
Also by \eqnok{prox_step}, we have $h(x_k^l, x) \le l$. Using these observations
and Lemma~\ref{three_points}, we have}
\[
f(x^u_k) \le f(\tilde x^u_k) \le (1 - \alpha_k) f(x^u_{k-1}) + \alpha_k l
+ \frac{M}{1+\rho} \|\alpha_k(x_k - x_{k-1})\|^{1+\rho}, \ \ \forall k \ge 1.
\]
Subtracting $l$ from both sides of the above inequality, we obtain
\beq \label{rec_APL}
f(x^u_k) - l \le (1 - \alpha_k) [f(x^u_{k-1})-l]+ \frac{M}{1+\rho} \|\alpha_k(x_k - x_{k-1})\|^{1+\rho},
\ \ \forall k \ge 1.
\eeq
\textcolor{blue}{Using the above inequality and Lemma~\ref{tech_result_sum} (with $\delta_k = f(x^u_k) -l$, $w_k = 1-\alpha_k$,
$W_k = \gamma_k$ and $B_k =M\|\alpha_k(x_k - x_{k-1})\|^{1+\rho}/(1+\rho)$), we obtain}
\beqas
f(x^u_k) - l &\le& (1 -\alpha_1) \gamma_k [f(x^u_{0})-l]
+ \frac{M}{1+\rho} \gamma_k \sum_{i=1}^k \gamma_i^{-1} \|\alpha_i(x_i - x_{i-1})\|^{1+\rho}\\
&\le& (1 -\alpha_1) \gamma_k [f(x^u_{0})-l] + \frac{M}{1+\rho} \|\Gamma_k\|_\frac{2}{1-\rho} 
\left[\sum_{i=1}^{k} \|x_i - x_{i-1}\|^2\right]^\frac{1+\rho}{2}, \ \ \forall k \ge 1
\eeqas
where the last inequality follows from H\"{o}lder's inequality.
The above conclusion together with \eqnok{cluster_APL} then imply that \eqnok{Recursion_APL} holds.

Now, denote $K = K_{APL}(\epsilon)$ and suppose that condition \eqnok{APL_policy} holds. Then
by \eqnok{Recursion_APL}, \eqnok{APL_policy}, \eqnok{def_cal_DX} and \eqnok{diam_X}, we have
\[
f(x^u_K) - l \le \frac{c M}{1+\rho}\left[\frac{2 d_\w(x_K)}{\sigma_\w}\right]^\frac{1+\rho}{2} K^{-\frac{1+3\rho}{2}}
\le \frac{c M}{1+\rho}\Omega_{\w,X}^\frac{1+\rho}{2} K^{-\frac{1+3\rho}{2}}
\le \theta \beta \Delta_0 = \theta (\overline{f}_0 - l),
\]
where the last equality from the fact that
$l = \beta \, \underline{f}_0 + (1-\beta) \, \overline{f}_0
= \overline{f}_0 - \beta \Delta_0$. Hence, procedure~$\APLGAP$ must
terminate in step 3 of the $K$-th iteration.
\end{proof}

\vgap

In view of Theorem~\ref{APL_sum}, we discuss below a few possible selections of $\{\alpha_k\}$,
which satisfy condition~\eqnok{APL_policy} and thus guarantee the termination of
procedure~$\APLGAP$. It is worth noting that these selections of $\{\alpha_k\}$
do not rely on any problem parameters, including $M$, $\rho$ and $\Omega_{\w,X}$,
nor on any other algorithmic parameters, such as $\beta$ and $\theta$.

\begin{proposition} \label{APL-step}
Let $\gamma_k(\cdot)$ and $\Gamma_k(\cdot,\cdot)$, respectively, be defined in \eqnok{def_ABL_gamma}
and \eqnok{def_ABL_gamma1}.
\begin{itemize}
\item [a)] If $\alpha_k = 2/(k+1)$, $k = 1, 2, \ldots$,
then $\alpha_k \in (0,1]$ and relation \eqnok{APL_policy} holds with 
$
c = 2^{1+\rho} 3^{-\frac{1-\rho}{2}}.
$
\item [b)] If $\alpha_k$, $k = 1, 2, \ldots$, are recursively
defined by
\beq \label{stepsize2}
\alpha_1 = \gamma_1 = 1, \ \
\gamma_k = \alpha_k^2 = (1 - \alpha_k) \gamma_{k-1},
\ \ \forall \, k \ge 2,
\eeq
then we have $\alpha_k \in (0,1]$ for any $k \ge 1$. Moreover,
condition \eqnok{APL_policy} is satisfied with
$
c =  \frac{4} {3^{\frac{1-\rho}{2}}}.
$
\end{itemize}
\end{proposition}

\begin{proof}
We show part a) only since part b) follows directly from Proposition~\ref{ABL-step}.b) with $\lambda = 1$.
Denoting $\gamma_k \equiv \gamma_k(1)$ and $\Gamma_k \equiv \Gamma_k(1,\rho)$,
by \eqnok{def_ABL_gamma} and \eqnok{def_ABL_gamma1}, we have
\beq \label{rel_step1}
\gamma_k = \frac{2}{k (k+1)} \ \ \mbox{and} \ \
\gamma_k^{-1} \alpha_k^{1+\rho} = \left(\frac{2}{k+1}\right)^\rho k
\le 2^\rho k^{1-\rho}.
\eeq
Using \eqnok{rel_step1} and the simple observation that
$
\sum_{i=1}^k i^2 = k(k+1)(2k+1)/6 
\le k(k+1)^2/3,
$
we have
\beqas
\gamma_k \|\Gamma_k\|_{\frac{2}{1-\rho}} &\le& \gamma_k \left[\sum_{i=1}^k 
\left( 2^\rho i^{1-\rho}\right)^\frac{2}{1-\rho}
\right]^\frac{1-\rho}{2} 
= 2^\rho \gamma_k \left(
\sum_{i = 1}^k i^2 \right)^\frac{1-\rho}{2} 
\le 2^\rho \gamma_k \left[ \frac{k (k+1)^2}{3}\right]^\frac{1-\rho}{2}\\
&=& \left(2^{1+\rho} \, 3^{-\frac{1-\rho}{2}}\right) 
\left[k^{-\frac{1+\rho}{2}} (k+1)^{-\rho}\right]
\le \left(2^{1+\rho} \, 3^{-\frac{1-\rho}{2}}\right)  \, k^{-\frac{1+3\rho}{2}}.
\eeqas
\end{proof}

\vgap

In view of Lemma~\ref{obs_gap}.e) and the termination criterion of
procedure~$\APLGAP$, each call to this procedure can reduce the gap between a given
upper and lower bound on $f^*$ by a constant factor $q$ (see \eqnok{def_q}).
In the following APL method, we will iteratively call procedure~$\APLGAP$
until a certain accurate solution of problem \eqnok{cp} is found.

\vgap

\noindent {\bf The APL method:} 
\begin{itemize}
\item [] {\bf Input:} initial point $p_0 \in X$, tolerance $\epsilon > 0$ \textcolor{blue}{and algorithmic parameters $\beta, \theta \in (0,1)$}.
\item [0)] Set $p_1 \in \Argmin_{x \in X} h(p_0, x)$, $\lb_1 = h(p_0, p_1)$
and $\ub_1 = f(p_1)$. Let $s=1$.
\item [1)] If $\ub_s - \lb_s \le \epsilon$, {\bf terminate};
\item [2)] Set $(p_{s+1}, \lb_{s+1}) = \APLGAP(p_s, \lb_s,\textcolor{blue}{ \beta, \theta})$ and $\ub_{s+1} = f(p_{s+1})$;
\item [3)] Set $s = s+1$ and go to step 1.
\end{itemize}

Similarly to the ABL method, whenever $s$ increments by $1$, we say that a phase of the APL method occurs.
Unless explicitly mentioned otherwise, an iteration of
procedure~$\APLGAP$ is also referred to as an iteration of the APL method.
The main convergence properties of the above APL method are summarized as follows.

\begin{theorem} \label{APL_theorem}
Let $M$, $\rho$, $\Omega_{\w,X}$ and $q$ be defined in 
\eqnok{smoothness}, \eqnok{diam_X} and \eqnok{def_q}, respectively.
Suppose that $\alpha_k\in (0,1]$, $k = 1, 2, \ldots$,
in procedure~$\APLGAP$ are chosen such that condition \eqnok{APL_policy} holds for some $c>0$.

\begin{itemize}
\item [a)] The number of phases performed by the APL method does not exceed
\beq \label{bnd_phase_APL}
\bar{S}(\epsilon) := \left \lceil \max \left\{0,
\log_{\frac{1}{q}}
 \frac{M {\Omega}_{\w,X}^\frac{1+\rho}{2}}{(1+\rho) \epsilon} \right\}
\right \rceil.
\eeq
\item [b)] The total number of iterations
performed by the APL method can be bounded by
\beq \label{bound_APL_iter}
\bar{S}(\epsilon)+
\frac{1}{1 - q^\frac{2}{1+3\rho}}
\left(
\frac{c M \Omega_{\w,X}^\frac{1+\rho}{2}}{\beta \theta (1+\rho) \epsilon}
\right)^\frac{2}{1+3\rho} .
\eeq
\end{itemize}
\end{theorem}

\begin{proof}
Denote $\delta_s \equiv \ub_s - \lb_s$, $s \ge 1$. 
Without loss of generality, we assume that $\delta_1 > \epsilon$, 
since otherwise the statements are obviously true.
By Lemma~\ref{obs_gap}.e) and the origin of $\ub_s$ and $\lb_s$, we have 
\beq \label{de_delta_APL}
\delta_{s+1} \le q \delta_s, \ \ s \ge 1.
\eeq
Also note that, by \eqnok{bound_delta0} and \eqnok{diam_X}, we have
\[
\delta_{1} \le \frac{M \|p_1 - p_0\|^{1+\rho}}{1+\rho} \le \frac{M \Omega_{\w,X}^\frac{1+\rho}{2}}{1+\rho}.
\]
The previous two observations then clearly imply that the number of phases performed by
the APL method is bounded by \eqnok{bnd_phase_APL}. 
 
We now bound the total number of iterations performed by the APL method. 
Suppose that procedure $\APLGAP$ has been called
$\bar s$ times for some $1\le \bar s \le \bar{S}(\epsilon)$.
It follows from \eqnok{de_delta_APL} that
$\delta_{s} > \epsilon q^{s-\bar s}$, $s=1, \ldots, \bar s$, since
$\delta_{\bar s} > \epsilon$ due to the origin of $\bar s$.
Using this observation, we obtain
\[
\sum_{s=1}^{\bar s} \delta_{s}^{-\frac{2}{1+3\rho}}
 < \sum_{s=1}^{\bar s} \frac{ q^{\frac{2}{1+3\rho}(\bar s  -s)}}{\epsilon^{\frac{2}{1+3\rho}} }
 = \sum_{t=0}^{\bar s-1} \frac{ q^{\frac{2}{1+3\rho}t}}{\epsilon^{\frac{2}{1+3\rho}} }
   \le \frac{1}{(1 - q^\frac{2}{1+3\rho}) \epsilon^\frac{2}{1+3\rho}
 }.
\]
Moreover, by Theorem~\ref{APL_sum}, the total number of iterations performed by the APL method is bounded by
\[
\sum_{s=1}^{\bar s} K_{APL}(\delta_s) \le \bar s 
+\sum_{s=1}^{\bar s} \left(
\frac{c M \Omega_{\w,X}^\frac{1+\rho}{2}}{\beta \theta (1+ \rho)\delta_s}
\right)^\frac{2}{1+3\rho} \\
\]
Our result then immediately follows by combining the above two inequalities.
\end{proof}

\vgap

Clearly, in view of Theorem~\ref{APL_theorem}, the APL method
also uniformly achieves the optimal complexity for solving smooth, weakly
smooth and nonsmooth CP problems. In addition, its iteration
cost can be significantly smaller than that of the APL method.

\setcounter{equation}{0}
\section{Level methods for solving composite and structured nonsmooth CP problems} \label{sec-app}
In this subsection, we discuss two possible ways to generalize the APL method
developed in Section~\ref{sec-APL}.
More specifically, we discuss a relatively easy extension of the APL method
for solving an important class of composite problems in
Section~\ref{sec-comp}, and present a more involved generalization of this method
for solving a certain class of saddle point problems in Section \ref{sec-saddle}.
Throughout this section, we assume that
$\|\cdot\|$ is an arbitrary norm in its embedded Euclidean space (not necessarily the one
associated with the inner product) \textcolor{blue}{ and use $\|\cdot\|_2$ to denote the $l_2$ norm}.

\subsection{Composite CP problems} \label{sec-comp}
In this subsection, we consider the CP problem \eqnok{cp} with $f$ given by:
\beq \label{def_comp}
f(x):= \Psi(\phi(x)),
\eeq
where the outer function $\Psi: \bbr^m \to \bbr$ is 
Lipschitz continuous and convex, and the inner function, given by
$\phi(x) = \left(\phi_1(x), \ldots, \phi_m(x)\right)$,
is an $m$-dimensional vector-function with Lipschitz continuous and convex 
components $\phi_i$, $i = 1, \ldots, m$. For the sake of notational convenience,
we refer to this class of problems as problem \eqnok{cp}-\eqnok{def_comp}.
We assume that the structure 
of $\Psi$ is relatively simple in comparison with $\phi$ 
(see Examples 1-4) and known to the iterative schemes 
for solving \eqnok{def_comp}, while the inner functions 
$\phi_i$, $i = 1, \ldots, m$, are represented by the black-box first-order 
oracles. These first-order oracles return, given an input point $x \in X$, 
the function values $\phi_i(x)$ and (sub)gradients $\phi'_i(x)$.
The following three additional assumptions are made about $\phi$ and $\Psi$.

\begin{assumption} \label{assmp1}
$\exists$ $\rho_i \in [0,1]$ and $M_i \ge 0$ such that:
\beq \label{smoothness_level}
\|\phi_i'(x) - \phi_i'(y) \|_* \le M_i \|x - y\|^{\rho_i},
\ \ \ \forall \, x, y \in X.
\eeq
\end{assumption}

Observe that relation \eqnok{smoothness_level} holds with $\rho_i = 1$, $0$ and
$(0,1)$, respectively, for smooth, nonsmooth and weakly smooth components $\phi_i$ 
(c.f. \cite{NemNes85-1,Nest88-1,DeGlNe10-1}). 
Clearly, if $M_i = 0$ for some $1 \le i \le m$, then the component
$\phi_i$ must be affine. Otherwise, $\phi_i$ must be nonlinear. 
To fix the notation, let us assume throughout this subsection that,
for a given $1 \le m_0 \le m$, the 
first $m_0$ components of $\phi$ are nonlinear, 
i.e., $M_i > 0$ for any $1 \le i \le m_0$, while the remaining $m-m_0$ 
components are affine, i.e., $M_i = 0$ for any $m_0 + 1\le i \le m$. 
We make the following assumption regarding the monotonicity of $\Psi$ 
with respect to these nonlinear components.
 
\begin{assumption} \label{assump2}
The map
\[
y_i \mapsto \Psi(y_1, \cdots, y_i, \cdots, y_m)
\]
is monotonically nondecreasing for any $1 \le i \le m_0$. 
\end{assumption}

In addition, we make a certain ``Lipschitz-continuity'' assumption
about $\Psi$.

\textcolor{blue}{
\begin{assumption} \label{assump3}
There exists $M_0 \in [0, \infty)$ such that
\beq \label{def_M0}
M_0 := \sup_{y \in \bbr^m, \delta \in \bbr^m_+}
\left\{\frac{\Psi(y+\delta) - \Psi(y)}{\|\delta\|_1}:
\delta_i = 0, \forall \, m_0 + 1\le i \le m.
\right\}.
\eeq
\end{assumption}}

\vgap

Many CP problems can be written in the form of problem \eqnok{cp}-\eqnok{def_comp}. 
We give a few interesting examples as follows.

\begin{example}{\bf Nonsmooth, weakly smooth and smooth problems.}
Let $m=1$ and $\Psi(y) = y$. Then, problem \eqnok{cp}-\eqnok{def_comp} covers the usual nonsmooth,
weakly smooth and smooth CP problems, for which condition \eqnok{smoothness_level}
is satisfied with $\nu = 0$, $\nu \in (0,1)$ and
$\nu = 1$, respectively. 
\end{example}

\begin{example}{\bf Minimax problems.}
Let $\Psi(y) = \max\{y_1, \ldots, y_m\}$. With this outer function,
problem \eqnok{cp}-\eqnok{def_comp} becomes the minimax problem to minimize the maximum
of a finite number of convex functions. It can be 
used, for example, to solve a system of smooth convex inequalities
$\phi_i(x) \le 0$, $i = 1, \ldots, m$, where $\phi_i(x)$ are convex functions
satisfying \eqnok{smoothness_level} with $\nu_1 = \ldots = \nu_m = 1$.
It can also be used to solve a system of mixed smooth and nonsmooth 
convex inequalities if $\nu_i =0$ or $1$, $i = 1, \ldots,m$.
\end{example} 

\begin{example} \label{comp} {\bf Composite smooth and nonsmooth problems.}
Consider $\min_{x \in X} \psi(x) = \phi_1(x)+\phi_2(x)$, where 
$\phi_1$ is a smooth component and $\phi_2$ is
a nonsmooth component. Clearly, we can write the problem in the form 
of \eqnok{def_comp} by setting $\phi(x) = (\phi_1(x), \phi_2(x))$ and 
$\Psi(y_1, y_2) = y_1 + y_2$. For this problem, we have $\rho_1 = 1$ and $\rho_2=0$.
The applications can be found, for example, in certain penalization approaches 
for solving nonsmooth CP problems \cite{Lan10-3}.
\end{example}

\begin{example} {\bf Regularized problems.}
Consider the problem $\min_{x \in X} \phi_1(x) + \rho r(x)$, where
$\phi_1$ is a smooth convex function with Lipschitz continuous gradient 
and $r(x)$ is a continuous, nonnegative, usually nonsmooth 
convex function. \textcolor{blue}{Clearly, this problem is a special case of Example~\ref{comp}.
However, sometimes we may want to keep the regularization term
$\rho r(x)$ in the definition of $\Psi$, so that this term will not be
linearized when defining the cutting plane model (c.f. \eqnok{comp_lower_merit}).}
For this purpose, we can put this problem in the form of \eqnok{cp}-\eqnok{def_comp}
by setting $\phi(x) = (\phi_1(x), x) \in \bbr^{n+1}$ and $\Psi(y, x) = y +\rho r(x)$.
Note that if $r(x) = \|x\|_1$ and $\phi_1(x) = \textcolor{blue}{\|Ax - b\|_2^2}$, this
problem becomes the well-known \textcolor{blue}{$l_1$ regularized least squares problem}.
\end{example}

\vgap

Since problem \eqnok{cp}-\eqnok{def_comp} covers nonsmooth, weakly smooth
and smooth CP as certain special cases and $m_0$ is a given constant, in view
of \cite{nemyud:83,NemNes85-1,Nest88-1,DeGlNe10-1}, a lower bound on the 
iteration complexity for solving this class of generalized composite problems 
is given by
\beq \label{lb_comp}
\max_{i=1, \ldots, m_0} \left(\frac{M_i} {\epsilon}\right)^\frac{2}{1+3\rho_i}.
\eeq

The composite CP problem described above generalizes a few other 
composite CP problems existing in the literature
(see, for example, Nesterov~\cite{Nest04,Nest07-1}, Tseng \cite{tseng08-1}, 
Lewis and Wright \cite{LewWri09-1}, Sagastiz\'{z}bal~\cite{Saga11-1} and Nemirovski~\cite{Nem94}). 
Our goal in this subsection is to show that, by properly modifying the APL method,
we can uniformly achieve the optimal complexity for solving \eqnok{def_comp} without requiring any global 
information about the inner functions $\phi_i$, such as, 
the smoothness levels $\rho_i$ and the Lipschitz constants $M_i$, 
for all $i = 1, \ldots,m_0$.

\vgap

Observing that the structure of $\Psi$ is known, we will replace the
cutting plane $h(\cdot, \cdot)$ used in procedure~$\APLGAP$
with the support function given by
\beq \label{comp_lower_merit}
h_\Psi(z,x) := \Psi(\phi(z) +  \phi'(z) (x -z)),
\eeq
where $\phi'(z) d := (\langle \phi_1'(z), d\rangle; \ldots; \langle \phi_k'(z), d\rangle)$
and $\phi_i'(z) \in \partial \phi_i(z)$ for any $i = 1, \ldots, m$.
We will refer to the APL method using the above support function as
{\sl the modified APL method}. Lemma~\ref{lw_merit_comp} describes some basic properties of 
$h_\Psi$.



\begin{lemma} \label{lw_merit_comp}
Let $h_\Psi(\cdot, \cdots)$ be defined in \eqnok{comp_lower_merit}.
We have
\beq \label{smoothness_comp}
h_\Psi(z,x) \le \Psi(\phi(x)) \le h_\Psi(z,x) + M_0 \sum_{i=1}^{m_0} 
\frac{M_i}{1+\rho_i} \|x - z\|^{1+\rho_i},
\eeq
where $M_i$, $\rho_i$ and $M_0$ are given by \eqnok{smoothness_level}
and \eqnok{def_M0}, respectively.
\end{lemma}

\begin{proof}
Let us denote $o_i \equiv M_i \|x-z\|^{1+\rho_i} / (1 + \rho_i)$, $i = 1, \ldots, m_0$,
and $o = \left(o_1, o_2, \ldots, o_{m_0}, 0, \ldots, 0\right)$.
Clearly, we have $h_\Psi(z,z) = \Psi(\phi(z)) = f(z)$ for any $z \in X$.
Moreover, it follows from \eqnok{smoothness} and \eqnok{smoothness_level} that
\[
\phi_i(z) + \langle \phi_i(z), x - z\rangle
\le \phi_i(x) \le \phi_i(z) + \langle \phi'_i(z), x - z\rangle
+ o_i, 
\]
for any $i = 1, \ldots, m_0$ and that
$\phi_i(x) = \phi_i(z) + \langle \nabla \phi_i(z), x - z\rangle$
for any $i = m_0 + 1, \ldots, m$. Using these observations, 
Assumption~\ref{assump2} and the definition of $M_0$
in \eqnok{def_M0}, we have
\beqas
\Psi(\phi(z) + \phi'(z) (x-z)) &\le& \Psi(\phi(x)) 
\le \Psi( \phi(z) + \phi'(z) (x-z) + o) \\
&\le& \Psi(\phi(z) + \phi'(z) (x-z)) + M_0 \sum_{i=1}^{m_0} 
o_i.
\eeqas
\end{proof}

\vgap

Observe that the second relation in \eqnok{smoothness_comp}
depends on $M_0$. In view of \eqnok{def_M0}, it can be easily
seen that $M_0 = 1$ for Examples 1, 2, 3 and 4 mentioned above.
We are now ready to describe the main convergence properties
of the aforementioned modified APL method.

\begin{theorem} \label{theorem_comp}
Consider the modified APL method applied
to problem \eqnok{cp}-\eqnok{def_comp} where we replace $h(\cdot, \cdot)$
by $h_\Psi(\cdot, \cdot)$. Suppose that $\{\alpha_k\}$ in procedure~$\APLGAP$ are chosen
such that condition~\eqnok{APL_policy} for all $\rho\in [0,1]$ for some $c>0$.
Let  $\Omega_{\w,X}$, $q$, $M_i$, $\rho_i$ and $M_0$ be defined in \eqnok{diam_X},
\eqnok{def_q},  \eqnok{smoothness_level} and \eqnok{def_M0} respectively.
\begin{itemize}
\item [a)] The number of phases performed by the modified APL method does not exceed:
\beq \label{bnd_phase_comp}
{\cal S}_\Psi(\epsilon) := \left \lceil \max\left\{0, \log_{\frac{1}{q}} \left(\sum_{i=1}^{m_0} 
\frac{M_0 M_i}{(1+\rho_i) \, \epsilon}  \Omega_{\w,X}^\frac{1+\rho_i}{2}\right) \right \}
\right \rceil.
\eeq
\item [b)] The total number iterations performed by the modified APL method can be bounded by
\beq \label{bnd_iter_comp}
{\cal S}_\Psi(\epsilon) + \sum_{i=1}^{m_0} \left[
\frac{m_0 c M_0 M_i }
{\beta \theta \, (1+\rho_i) (1- q^\frac{2}{1+ 3\rho_i}) \, \epsilon} 
\Omega_{\w,X}^\frac{1+\rho_i}{2}
\right]^\frac{2}{1 + 3\rho_i}.
\eeq
\end{itemize}
\end{theorem}

\begin{proof}
The proof of the result
is similar to that of Theorem~\ref{APL_theorem}
by noticing the difference between relations~\eqnok{smoothness} and \eqnok{smoothness_comp},
and hence the details are skipped.
\end{proof}

\vgap

We now add a few comments about the results obtained
in Theorem~\ref{theorem_comp}. Firstly, if there exists only one nonlinear component in
the inner function $\phi(\cdot)$, i.e., $m_0=1$, then the bound in \eqnok{bnd_phase_comp}
reduces to the bound established for the original APL method, i.e., \eqnok{bound_APL_iter}.
Moreover, for a given $m_0 > 1$, we can see from \eqnok{lb_comp} and \eqnok{bnd_iter_comp} 
that the complexity bound in \eqnok{bnd_iter_comp} is optimal,
 up to a constant factor depending on $m_0$, for solving the 
composite CP problems. Finally, as shown by the following result,
in certain special cases, one can improve the dependence of the
iteration-complexity bound on the number of components of $\phi$. 

\begin{corollary} \label{cor_comp}
Suppose that $\rho_1 = \rho_2 = 
\ldots = \rho_{m_0}$ in Assumption~\ref{assmp1}.
Let us denote
\beq \label{def_tilde_M}
\tilde{M} := \sup\limits_{y \in \bbr^m, t > 0} 
\frac{\Psi(y+t \delta_M) - \Psi(y)}{t}, \ \
\delta_M := (M_1, \ldots, M_{m_0}, 0, \ldots, 0).
\eeq
Then, the total number of phases and iterations performed by
the above modified APL method applied to problem~\eqnok{cp}-\eqnok{def_comp} can be bounded by
\eqnok{bnd_phase_APL} and \eqnok{bound_APL_iter}, respectively,
with $M = \tilde{M}$ and $\rho = \rho_1$.
\end{corollary}

\begin{proof}
Similarly to Lemma~\ref{lw_merit_comp}, we can show that
\[
h_\Psi(z,x) \le \Psi(\phi(x)) \le h_\Psi(z,x) + \frac{\tilde M}{1+\rho_1} \|x - z\|^{1+\rho_1}.
\]
The rest of the proof is 
similar to that of Theorem~\ref{APL_theorem} and
hence the details are skipped.
\end{proof}

\vgap

Consider a special case of problem \eqnok{cp}-\eqnok{def_comp} where $m = m_0$, 
$\rho_1 = \ldots = \rho_m$ and $\Psi(y) = \max_{1 \le i \le m} y_i$
(see Example 2). We can easily see from \eqnok{def_tilde_M} that
$\tilde{M} = \max_{1\le i \le m_0} M_i$ and hence that, by Corollary \eqnok{cor_comp},
the iteration-complexity bound of the APL method does not depend on the 
number of components in the inner function $\phi(\cdot)$.

\subsection{Structured nonsmooth CP problems} \label{sec-saddle}
In this subsection, we present a new BL type method for
solving a class of structured nonsmooth CP problems that has recently 
been studied by Nesterov (c.f. \cite{Nest05-1,Nest05-2}). Consider problem \eqnok{cp} with
$f$ given by
\beq \label{def_saddle}
f(x) := \hat{f}(x) + F(x),
\eeq
where $\hat{f}: X \to \bbr$ is a simple Lipschitz continuous convex function and
\beq \label{nonsmooth1}
F(x) := \max\limits_{y \in \cY} \left\{ \langle A x , y \rangle - \hat{g}(y)\right\}.
\eeq
Here, $ Y \subseteq \bbr^m$ is a compact convex set, $\hat{g}: \cY \to \bbr$ 
is a continuous convex function on $\cY$ and
$A$ denotes a linear operator from $\bbr^n$ to $\bbr^m$.
Observe also that problem \eqnok{cp}-\eqnok{def_saddle} 
can be written in an adjoint form:
\beq \label{dual}
\max_{y \in \cY} \left\{g(y) := -\hat{g}(y) + G(y)  \right\},
\ \ \
G(y) := \min_{x \in \cX} \left\{ \langle A x , y \rangle + \hat{f}(x)
\right\}.
\eeq

%
While the function $F$ given by \eqnok{nonsmooth1} 
is a nonsmooth convex function in general, Nesterov in an important work~\cite{Nest05-1}
shows that it can be closely approximated by a class of smooth 
convex functions. We now briefly describe Nesterov's smoothing 
scheme as follows. Let $v(y)$ be a prox-function of $Y$
with modulus $\sigma_v$ and prox-center $c_v = \argmin_{y \in Y} v(y)$. 
Also let us denote
\[
V(y) := v(y) - v(c_v) - \langle \nabla v(c_v), y - c_v \rangle,
\]
and, for some $\eta>0$,
\beqa 
F_\eta(x) &:=& \max\limits_{y}\left\{\langle A x, y \rangle - 
	\hat{g}(y)-\eta \, V(y): \ y \in \cY \right\}, \label{sm-approx}\\
f_\eta(x) &:=& \hat{f}(x) + F_\eta(x). \label{def_psi_eta}
\eeqa
It is shown in \cite{Nest05-1}
that $F_{\eta}(\cdot)$ has Lipschitz-continuous gradient with constant 
\beq \label{new_ls}
{\cal L}_\eta \equiv {\cal L}(F_\eta) := \frac{\|A\|^2}{\eta\sigma_v},
\eeq
where $\|A\|$ denote the operator norm of $A$.
Moreover, the ``closeness'' of $F_\eta(\cdot)$ to $F(\cdot)$ depends linearly on the 
parameter $\eta$. In particular, we have, for every $x \in \setU$,
\beq \label{closeness}
F_{\eta}(x) \le F(x) \le F_{\eta}(x) + \eta \, \cD_{v,Y},
\eeq
and, as a consequence,
\beq \label{closeness1}
f_{\eta}(x) \le f(x) \le f_{\eta}(x)+\eta \, {\cal D}_{v,Y},
\eeq
where $\cD_{v,Y}$ given by \eqnok{def_cal_DX}.

\vgap

Nesterov shows in \cite{Nest05-1} that one can obtain an 
$\epsilon$-solution of problem \eqnok{cp}-\eqnok{def_saddle}
in at most ${\cal O} (1 / \epsilon)$ iterations,
by applying a variant of his optimal smooth method~\cite{Nest83-1,Nest05-1}
to $\min_{x \in X} f_\eta(x)$, for a properly chosen $\eta > 0$. 
This result is significantly better than the iteration-complexity for the black-box
nonsmooth convex optimization techniques applied to \eqnok{def_saddle}. 
However, to implement Nesterov's approximation scheme, it is necessary to know a 
number of problem parameters a priori, including
$\|A\|$, $\sigma_v$ and $\cD_{v,Y}$, and the total number of
iterations $N$. To eliminate the requirement that $N$ should 
be given in advance, Nesterov in \cite{Nest05-2} presented an excessive gap 
procedure where the above smoothing technique is applied to both the primal 
and dual problem \eqnok{dual}. However, to apply the excessive gap procedure, one needs 
to know a few more parameters, including $\|A\|$, $\sigma_v$, $\cD_{v,Y}$, $\sigma_\w$ and $\cD_{\w,X}$, 
where $\sigma_\w$ is the modulus of a given prox-function $\w$ of $X$ and $\cD_\w$ 
is defined in \eqnok{def_cal_DX}. In \cite{Nem05-1}, Nemirovski proposed 
a prox-method with ${\cal O}(1/\epsilon)$ iteration-complexity bound 
for solving a slightly more general class of CP problems than 
\eqnok{def_saddle}. To attain the best possible iteration complexity 
in \cite{Nem05-1}, it is still necessary to know the parameters $\|A\|$, $\sigma_v$, 
$\cD_{v,Y}$, $\sigma_\w$ and $\cD_{\w,X}$ explicitly. 
One possible approach for solving problem \eqnok{cp}-\eqnok{def_saddle} would be to apply the
APL method, which is shown to be optimal for smooth
convex optimization, to the smooth approximation $\min_{x\in X} f_\eta(x)$ for some 
$\eta > 0$, similarly to Nesterov's smoothing scheme~\cite{Nest05-1}. Note however, 
that this approach would require both the number of iterations (or the target accuracy)
and the problem parameter ${\cal D}_v$ (see \eqnok{closeness1}) given a priori. 

Our goal in this section is to present a completely problem parameter-free smoothing technique,
namely: the {\sl uniform smoothing level (USL) method},
obtained by properly modifying the APL method in Section~\ref{sec-APL}. 
In the USL method, the smoothing parameter $\eta$ is adjusted dynamically during their execution
rather than being fixed in advance. Moreover, an
estimate on the value of ${\cal D}_v$ can be provided automatically. 
We start by describing the USL gap reduction procedure, denoted by $\USLGAP$,
which will be iteratively called by the USL method. Specifically, for a given search point $p$, a lower bound $\lb$ on $f^*$ 
and an initial estimate $\tilde{D}$ on ${\cal D}_{v,Y}$, 
procedure~$\USLGAP$ will either compute a new search point $p^+$ and a new 
lower bound $\lb^+$ satisfying $f(p^+) - \lb^+ \le q \, [f(p) - \lb]$ for some $q \in (0,1)$, or
provide an updated estimate $\tilde{D}^+$ on ${\cal D}_{v, Y}$ in case the current estimate
$\tilde{D}$ is not accurate enough.

\vgap

\noindent{\bf The USL gap reduction procedure:  $(p^+, \lb^+, \tilde{D}^+) = \USLGAP(p, \lb, \tilde{D}, \textcolor{blue}{\beta, \theta})$} 
\begin{itemize}
\item [0)] Set $x^u_0 = p$, $\overline{f}_0 = f(x^u_0)$, $\underline{f}_0 = \lb$,
$l = \beta \underline{f}_0 + (1 - \beta) \overline{f}_0$, and
\beq \label{def_etas}
\eta := \theta (\overline{f}_0 - l) /(2 \tilde{D}).
\eeq
Also let $x_0 \in X$ and the initial localizer $X'_0$
be arbitrarily chosen, say $x_0 = p$ and $X'_0 = X$.
Set the prox-function $d(x) = \w(x) - [\w(x_0) + \langle \w'(x_0), x - x_0\rangle ]$.
Also let $k=1$. 
\item [1)] {\sl Update lower bound:} set $x^l_k = (1-\alpha_k) x^u_{k-1}+\alpha_k x_{k-1}$ and
\beq \label{lower_saddle}
h(x^l_k, x) = h_\eta(x^l_k,x) := \hat{f}(x) + F_\eta(x^l_k) + 
\langle \nabla F_\eta(x^l_k), x - x^l_k \rangle.
\eeq
Compute $\underline{f}_k$ according to \eqnok{def_lb}.
If $\underline{f}_k \ge l - \theta (l - \underline{f}_0)$,
then {\bf terminate} the procedure with $p^+ = x^u_{k-1}$, $\lb^+ = \underline{f}_k$,
and $\tilde{D}^+ = \tilde{D}$;
\item [2)] {\sl Update prox-center:} set $x_k$ according to \eqnok{prox_step};
\item [3)] {\sl Update upper bound:} set \textcolor{blue}{$\bar f_k = \min \{
\bar{f}_{k-1}, f(\alpha_k x_k + (1-\alpha_k) x^u_{k-1}) \}$}, and choose $x^u_k$
such that $f(x^u_k) = \bar f_k$.
Check the following two possible termination criterions:
\begin{itemize}
\item [3a)]
if 
$\overline{f}_k \le l + \theta (\overline{f}_0 - l)$,
{\bf terminate} the procedure with $p^+ = x^u_{k}$, $\lb^+ = \underline{f}_k$ and $\tilde{D}^+ = \tilde{D}$,
\item [3b)] Otherwise, if $
f_\eta(x^u_k) \le l + \frac{\theta}{2} (\overline{f}_0 - l)$,
{\bf terminate} the procedure with $p^+ = x^u_{k}$, $\lb^+ = \underline{f}_k$ and $\tilde{D}^+ = 2 \tilde{D}$;
\end{itemize}
\item [4)] {\sl Update localizer:} choose an arbitrary $X'_k$ such that
$
\underline{X}_k \subseteq X'_k \subseteq \overline{X}_k,
$
where $\underline{X}_k$ and $\overline{X}_k $ are defined in \eqnok{bound_X};
\item [6)] Set $k = k+1$ and go to Step 1.
\end{itemize} 

We notice that there are a few essential differences between procedure $\USLGAP$ described above
and procedure $\APLGAP$ in Section~\ref{sec-APL}. Firstly, in comparison with procedure $\APLGAP$, procedure $\USLGAP$ 
needs to use one additional input parameter, namely $\tilde D$, to define $\eta$ (see \eqnok{def_etas})
and hence the approximation function $f_\eta$ in \eqnok{def_psi_eta}.
 
Secondly, 
we use the support functions $h_\eta(x_k^l,x)$ of $f_\eta(x)$ defined in \eqnok{lower_saddle}
procedure~$\USLGAP$ rather than the cutting planes of $f(x)$ in procedure~$\APLGAP$.
Notice that by \eqnok{lower_saddle}, the convexity of $F_\eta$ and the first relation in \eqnok{closeness}, we have 
\beq \label{temp_rel}
h_\eta(x_k^l, x) \le \hat{f}(x)+ F_\eta(x) \le \hat{f}(x) + F(x) = f(x),
\eeq
which implies that the functions $h_\eta(x_k^l,x)$ underestimate $f$ everywhere on $X$.
Hence, $\underline{f}_k$ computed in step 1 of this procedure are indeed lower bounds of $f^*$.

Thirdly, there are three possibles ways to terminate procedure~$\USLGAP$.
Similarly to procedure~$\APLGAP$, if it terminates in step 1 and step 3a, then we say that significant progress has 
been made on the lower and upper bounds on $f^*$, respectively.
The new added termination criterion in step 3b will be used only if
the value of $\tilde D$ is not properly specified. 
We formalize these observations in the following simple result.


\begin{lemma} \label{simple_USL}
The following statements hold for procedure $\USLGAP$.
\begin{itemize}
\item [a)] If the procedure terminates in step 1 or step 3a,
we have $f(p^+) - \lb^+ \le q [f(p) - \lb]$, where $q$ is defined
in \eqnok{def_q};
\item [b)] If the procedure terminates in step 3b, then $\tilde D < {\cal D}_{v,Y}$.
\end{itemize}
\end{lemma}

\begin{proof}
The proof of part a) is the same as that of Lemma~\ref{obs_gap}.e) and we
only need to show part b). Observe that whenever step 3b occurs, we have
$\overline{f}_k > l + \theta (\overline{f}_0 - l)$ and $f_\eta(x^u_k) \le l + \frac{\theta}{2} (\overline{f}_0 - l)$.
Hence, 
\[
f(x^u_k) - f_\eta(x^u_k) = \overline{f}_k - f_\eta(x^u_k) > \frac{\theta}{2} (\overline{f}_0 - l),
\] 
which, in view of the second relation in \eqnok{closeness1}, then implies that
$\eta {\cal D}_{v,Y} > \theta (\overline{f}_0 - l) / 2$. Using this observation
and \eqnok{def_etas}, we conclude that $\tilde D < {\cal D}_{v,Y}$.
\end{proof}

\vgap

We observe that all the results in Lemma~\ref{obs_gap}.a-d) regarding the
execution of procedure~$\APLGAP$ also hold for procedure~$\USLGAP$.
In addition, similar to Theorem~\ref{APL_sum}, we 
establish below some important convergence properties of procedure~$\USLGAP$
by showing how the gap between $f(x^u_k)$ and the level $l$ decreases.

\begin{theorem} \label{USL_sum}
Let $\alpha_k \in (0,1]$, $k = 1, 2, \ldots$,
be given. Also let $(x^l_k, x_k, x^u_k) \in X \times X \times X$, $k \ge 1$, be the
search points, $l$ be the level and $d_\w(\cdot)$ be the prox-function, $\eta$ be the smoothing
parameter (see \eqnok{def_etas}) in procedure~$\USLGAP$. Then, we have
\beq \label{Recursion_USL}
f_\eta(x_k^u) -l \le  (1-\alpha_1) \gamma_k(1) \, [f_\eta(x_0^u) -l] +  
\frac{\|A\|^2 d_\w(x_k)}{\eta \sigma_\w  \sigma_v} \gamma_k(1) \, \|\Gamma_{k}(1, \rho)\|_\infty,
\eeq
for any
$k \ge 1$, where $\|\cdot\|_\infty$ denotes the $l_\infty$ norm, $\gamma_k(\cdot)$ and $\Gamma_k(\cdot, \cdot)$, 
respectively, are defined in \eqnok{def_ABL_gamma} and \eqnok{def_ABL_gamma1}.
In particular, if $\alpha_k \in (0,1]$, $k=1, 2, \ldots$, are chosen such that
condition~\eqnok{APL_policy} holds with $\rho = 1$ for some $c > 0$,
then the number of iterations performed
by procedure~$\APLGAP$ can be bounded by
\beq \label{USL_GAP_BND}
K_{USL}(\Delta_0, \tilde{D}) := \left \lceil 
\frac{2 \|A\|}{\beta \theta \Delta_0}
\sqrt{\frac{c {\cal D}_{\w,X} \tilde{D}}{\sigma_\w \sigma_v}}
\right \rceil,
\eeq
where ${\cal D}_{\w,X}$ is defined in \eqnok{def_cal_DX}.
\end{theorem} 

\begin{proof}
Note that, by \eqnok{temp_rel} and \eqnok{def_psi_eta}, we have $h_\eta(z,x) \le f_\eta(x)$ for any $z, x \in X$.
Moreover, by \eqnok{def_psi_eta}, \eqnok{lower_saddle} and the fact that
$F_\eta$ has Lipschitz continuous gradients with constant ${\cal L}_\eta$, we obtain
\[
f_\eta(x) - h_\eta(z,x) = F_\eta(x) - \left[
F_\eta(z) + \langle \nabla F_\eta(z), x - z \rangle
 \right] \le \frac{{\cal L}_\eta}{2} \|x - z\|^2 = \frac{\|A\|^2}{2 \eta \sigma_v} \|x - z\|^2,
\]
for any $z, x \in X$, where the last inequality follows from \eqnok{new_ls}.
In view of these observations, \eqnok{Recursion_USL} follows from 
an argument similar to the one used in the proof of \eqnok{Recursion_APL}
with $f = f_\eta$, $M = {\cal L}_\eta$ and $\rho = 1$.

Now using \eqnok{def_cal_DX}, \eqnok{APL_policy} (with $\rho = 1$), \eqnok{def_etas} and \eqnok{Recursion_USL},
we obtain
\beqas
f_\eta(x_k^u) -l &\le& \frac{\|A\|^2 d_\w(x_k)}{\eta \sigma_\w  \sigma_v} \gamma_k(1) \, \|\Gamma_{k}(1, \rho)\|_\infty 
\le \frac{c \|A\|^2 d_\w(x_k)}{\eta \sigma_\w  \sigma_v k^2} \\
&\le& \frac{c \|A\|^2 {\cal D}_{\w,X}}{\eta \sigma_\w  \sigma_v k^2} =
\frac{2 c \|A\|^2 {\cal D}_{\w,X} \tilde{D}}{\theta (\overline{f}_0 - l) \sigma_\w  \sigma_v k^2}.
\eeqas
Denoting $K = K_{USL}(\Delta_0, \tilde{D})$
and noting that
$\Delta_0 = \overline{f}_0 - \underline{f}_0 = (\overline{f}_0 - l)/\beta$,
we conclude from the previous inequality that
$
f_\eta(x_K^u) - l \le \theta (\overline{f}_0 - l) / 2.
$
This result together with \eqnok{closeness1} imply that, if $\tilde D \ge {\cal D}_{v,Y}$,
then $f(x_K^u) - l \le f_\eta(x_K^u) - l + \eta {\cal D}_{v,Y}
\le \theta (\overline{f}_0 - l)$. In view of these two observations and the termination criterions used in step 3, 
procedure~$\USLGAP$ must terminate in at most $K_{APL}(\Delta_0,\tilde{D})$ iterations.
\end{proof}

\vgap

In view of Lemma~\ref{simple_USL}, each call to procedure~$\USLGAP$ can reduce
the gap between a given upper and lower bound on $f^*$ by a constant factor $q$, or
update the estimate on ${\cal D}_{v,Y}$ by a factor of $2$.
In the following USL method, we will iteratively call procedure~$\USLGAP$
until a certain accurate solution is found.

\vgap

{\bf The USL method:}
\begin{itemize}
\item [] {\bf Input:} $p_0 \in X$, tolerance $\epsilon > 0$, initial 
estimate $Q_1 \in (0, {\cal D}_v]$ \textcolor{blue}{and algorithmic parameters $\beta, \theta \in (0,1)$}.
\item [1)] Set 
\beq \label{ini_lb_saddle}
p_1 \in \Argmin_{x \in X} \left\{h_0(p_0, x) 
:= \hat{f}(x) + F(p_0) + \langle F'(p_0), x - p_0 \rangle \right\},
\eeq
$\lb_1 = h_0(p_0, p_1)$ and $\ub_1 := \min \{ f(p_1), f(\tilde p_1)\}$.
Let $s = 1$.
\item [2)] If $\ub_s - \lb_s \le \epsilon$, {\bf terminate};
\item [3)] Set $(p_{s+1}, \lb_{s+1}, Q_{s+1}) = \USLGAP(p_s, \lb_s, Q_s, \textcolor{blue}{\beta, \theta})$
and $\ub_{s+1} = f(p_{s+1})$;
\item [4)] Set $s=s+1$ and go to step 1. 
\end{itemize}

We now make a few remarks about the USL method described above. Firstly, 
each phase $s$, $s \ge 1$, of the USL method is associated with 
an estimation $Q_s$ on ${\cal D}_{v,Y}$, and $Q_1 \in (0, {\cal D}_{v,Y}]$ is a given 
input parameter. Note that such a $Q_1$ can be easily obtained by the definition of ${\cal D}_{v,Y}$.
Secondly, we differentiate two types of phases: a phase is called {\sl significant}
if procedure~$\USLGAP$ terminates in step 1 or step 3a,
otherwise, it is called {\sl non-significant}. Thirdly, 
In view of Lemma~\ref{simple_USL}.b), 
if phase $s$ is non-significant, then we must have $Q_s \le {\cal D}_{v,Y}$.
In addition, using the previous observation, and
the facts that $Q_1 \le {\cal D}_{v,Y}$ and that $Q_s$ can be increased by a 
factor of $2$ only in the non-significant phases,
we must have $Q_s \le 2 {\cal D}_{v,Y}$ for all significant phases.

Before establishing the complexity of the above USL method,
we first present a technical result
which will be used to provide a convenient estimate
on the gap between the initial lower and upper bounds on $f^*$.
The proof of this result is given in the Appendix.

\begin{lemma} \label{tech_result}
Let $F$ be defined in \eqnok{nonsmooth1} and
$v$ be a prox-function of $Y$ with modulus $\sigma_v$.
We have
\beq \label{Lips_saddle}
F(x_0) - F(x_1) - \langle F'(x_1), x_0 - x_1 \rangle \le 
2 \left( \frac{2 \|A\|^2 {\cal D}_{v,Y}}{\sigma_v}\right)^\frac{1}{2} \|x_0 - x_1\|,
\ \ \ \forall x_0, x_1 \in \bbr^n,
\eeq
where $F'(x_1) \in \partial F(x_1)$ and ${\cal D}_{v,Y}$ is defined
in \eqnok{def_cal_DX}.
\end{lemma}

We are now ready to show the main convergence results for the USL method.

\begin{theorem} \label{theorem_saddle}
Suppose that $\alpha_k \in (0,1]$, $k = 1,2, \ldots$, in procedure $\USLGAP$
are chosen such that condition~\eqnok{APL_policy} holds with $\rho = 1$ for some $c > 0$.
The following statements hold for the USL method applied to 
problem \eqnok{cp}-\eqnok{def_saddle}:
\begin{itemize}
\item [a)] the number of non-significant phases is bounded by
$\tilde{\cal S}_F(Q_1) := \left \lceil \log {\cal D}_v/Q_1
\right \rceil
$, and the number of significant phases is bounded by
${\cal S}_F(\epsilon) \equiv {\cal S}(4 \bar \Delta_F, \epsilon, q)$, where
${\cal S}(\cdot,\cdot,\cdot)$ is defined in
\eqnok{bnd_phase_comp} and
\beq \label{ini_gap_saddle}
\bar \Delta_F := \|A\| \sqrt{ \frac{ 
{\cal D}_{\w,X} {\cal D}_{v,Y}}{\sigma_\w \sigma_v}}.
\eeq
\item [b)] the total number of gap reduction iterations
performed by the USL method does not exceed 
\beq \label{iter_complexity1_u}
{\cal S}_F(\epsilon) + \tilde{\cal S}_F(Q_1) +
\frac{\tilde{c} \bar \Delta_F}{\epsilon},
\eeq 
where $\tilde{c} := 2 [\sqrt{2}/(1-q) + \sqrt{2}+1] \sqrt{c} / \beta \theta $.
\end{itemize}
\end{theorem}

\begin{proof}
Denote $\delta_s \equiv \ub_s - \lb_s$, $s \ge 1$. 
Without loss of generality, we assume that $\delta_1 > \epsilon$, 
since otherwise the statements are obviously true.
The first claim in part a) immediately follows from the facts that
a non-significant phase can occur
only if $Q_1 \le {\cal D}_v$ due to Lemma~\ref{simple_USL}.b)
and that $Q_s$, $s\ge 1$, is increased by a factor of $2$
in each non-significant phase. 
In order to show the second claim in part a), we first bound the initial optimality
gap $\ub_1 - \lb_1$. By the convexity of $F$, \eqnok{def_psi_eta} and \eqnok{ini_lb_saddle}, 
we can easily see that $\lb_1 \le f^*$. Moreover, we conclude from 
\eqnok{def_psi_eta}, \eqnok{Lips_saddle} and \eqnok{ini_lb_saddle} that 
\beqas
\ub_1 - \lb_1 &\le& f(p_1) - \lb_1 = 
F(p_1) -  F(p_0) - \langle F'(p_0), p_1 - p_0 \rangle\\
&\le&  2 \left( \frac{2 \|A\|^2 
{\cal D}_{v,Y}}{\sigma_v}\right)^\frac{1}{2} \|p_1 - p_0\|
\le 4 \bar \Delta_F,
\eeqas
where the last inequality follows from \eqnok{diam_X}.
Using this observation and Lemma~\ref{simple_USL}.a), we
can easily see that the number of significant phases is bounded by ${\cal S}_F(\epsilon)$.

We now show that part b) holds.
Let $B=\{b_1, b_2, \ldots, b_k\}$ and $N = \{n_1, n_2, \ldots, n_m\}$, 
respectively, denote the set of indices of
the significant and non-significant phases. 
Note that 
$
\delta_{b_{t+1}} \le q \, \delta_{b_{t}} ,  \, t \ge 1,
$
and hence that $\delta_{b_{t}} \ge q^{t-k} \delta_{b_{k}}
> \epsilon q^{t-k}$, $1 \le t \le k$.
Also observe that $Q_{b_t} \le 2 {\cal D}_{v,Y}$ (see the remarks
right after the statement of the USL method).
Using these observations and Theorem~\ref{USL_sum}, we conclude that
the total number of iterations performed in the significant phases
is bounded by 
\beqa
\sum_{t=1}^k K_{USL}(\delta_{b_t}, Q_{b_t}) &\le& \sum_{t=1}^k K_{USL}(\epsilon q^{t-k}, 2 {\cal D}_{v,Y})
\le k + \frac{2 \|A\|}{\beta \theta \epsilon} 
\sqrt{\frac{2 \C_1 {\cal D}_{\w,X} {\cal D}_{v,Y}}{\sigma_\w \sigma_v}} \sum_{t=1}^k q^{k-t} \nn \\
&\le& {\cal S}_F + \frac{2 \|A\|}{\beta \theta (1-q) \epsilon} 
\sqrt{\frac{2 \C_1 {\cal D}_{\w,X} {\cal D}_{v,Y}}{\sigma_\w \sigma_v}}, \label{ttt1}
\eeqa
where the last inequality follows from part a) and the observation
that $\sum_{t=1}^k q^{k-t} \le 1/ (1-q)$. Moreover,
note that $\Delta_{n_r} > \epsilon$ for any $1 \le r \le m$ and
that $Q_{n_{r+1}} = 2 Q_{n_r}$ for any $1 \le r \le m$. 
Using these observations and Theorem~\ref{USL_sum},
we conclude that the total number of iterations performed in 
the non-significant phases is bounded by
\beqa
\sum_{r =1}^m K_{USL}(\delta_{n_r}, Q_{n_r}) &\le& 
\sum_{r =1}^m K_{USL}(\epsilon, Q_{n_r}) 
\le  m+ \frac{2 \|A\|}{\beta \theta \, \epsilon} 
\sqrt{\frac{\C_1 {\cal D}_{\w,X} Q_1}{\sigma_\w \sigma_v}} \sum_{r=1}^m 2^\frac{r-1}{2}\nn \\
&\le& \tilde{\cal S}_F +  \frac{2 \|A\|}{\beta \theta \, \epsilon} 
\sqrt{\frac{\C_1 {\cal D}_{\w,X} Q_1}{\sigma_\w \sigma_v}} \sum_{r=1}^{\tilde{\cal S}_F} 2^\frac{r-1}{2}
\le \tilde{\cal S}_F + \frac{2 \|A\|}{(\sqrt{2}-1)\beta \theta \, \epsilon} 
\sqrt{\frac{\C_1 {\cal D}_{\w,X} {\cal D}_{v,Y}}{\sigma_\w \sigma_v}}. \label{ttt2}
\eeqa
Combining \eqnok{ttt1} and \eqnok{ttt2}, we obtain \eqnok{iter_complexity1_u}.
\end{proof}

\vgap

It is interesting to observe that, if $Q_1 ={\cal D}_{v,Y}$,
then there are no non-significant phases and the number of iterations performed by the USL method is simply
bounded optimally by \eqnok{ttt1}. In this case, we do not need to compute the value of $f_\eta(x^u_k)$
in step 3b. We refer to such a special case of the USL method as the basic 
smoothing level (BSL) method. It is interesting to note that, in view of Theorem~\ref{theorem_saddle},
the USL method still 
achieves the optimal complexity bound in \eqnok{iter_complexity1_u} even without a good 
initial estimate on ${\cal D}_{v,Y}$.

\setcounter{equation}{0}
\section{Numerical results} \label{sec-num}
Our objective in this section is to report some promising  
results obtained for the new BL type algorithms
developed in this paper. More specifically, 
we study in Sections \ref{sec-sdp} and \ref{sec-sp}, respectively, the application of these methods to
solve certain classes of semidefinite programming (SDP) and stochastic programming (SP)
problems. 

\subsection{A class of SDP problems} \label{sec-sdp}
In this subsection, we consider the classic SDP problem of 
\beq \label{prob_eig}
\min_{x \in X} \lambda_1\left( {\cal A}(x) \right),
\eeq
where $X \subseteq \bbr^n$ is a convex and compact set,
$\lambda_1: \bbr^{m\times m} \to \bbr$ denotes the
maximal eigenvalue of a symmetric matrix, 
\[
{\cal A}(x) = A_0 + \sum_{i=1}^n x_i A_i,
\]
and $A_i$, $i = 1, \ldots, n$, are given $m \times m$ symmetric 
matrices.  

One can solve problem \eqnok{prob_eig} by using interior-point methods.
However, due to the high iteration cost of interior-point methods, much 
effort has recently been directed to the development
of first-order methods for solving problem \eqnok{prob_eig}. 
Since problem \eqnok{prob_eig} is in general nonsmooth, 
one can use general nonsmooth convex optimization methods,
such as NERML (non-Euclidean restricted memory level) in \cite{BenNem00,BenNem05-1} or APL
in Section~\ref{sec-APL}. In particular, Let a symmetric matrix $A \in \bbr^{m \times m}$ be given.
It is well-known that the subdifferential of $\lambda_1$ at $A$ 
is given by
$
\partial \lambda_1(A) = \mbox{co}\left\{ 
u u^T: u^T u = 1, A u = \lambda_1(A) u \right\},
$
where $\mbox{co}(\cdot)$ denotes the convex hull. Hence, $\lambda_1$ is smooth (i.e.,
$\partial \lambda_1(A)$ is a singleton) if and only if the maximal 
eigenvalue of $A$ has multiplicity $1$. In comparison with the NERML algorithm, a nice feature of the APL
method is that it can automatically explore the local smoothness structures
of a particular problem instance, as
the objective function of \eqnok{prob_eig} may be differentiable along 
certain parts of the trajectory of the algorithm.
These methods, in the worst case,
require ${\cal O} ( 1/ \epsilon^2)$ iterations to find an $\epsilon$-solution
of problem~\eqnok{prob_eig}, and the major iteration costs of these methods consist of 
finding a maximal eigenvector $u_x$ of ${\cal A}(x)$ for a given $x \in X$ and
assembling the subgradient ${\cal A}^* u_x$, where ${\cal A}^*$ denotes the adjoint
operator of ${\cal A}$. 
It should be noted that other bundle type methods, such as
the spectral-bundle method by Helmberg and Rendl~\cite{HeRe00}, have
also been developed for solving problem \eqnok{prob_eig}. 
The spectral-bundle method is obtained by 
tailoring the well-known bundle method \cite{Kiw83-1,Kiw90-1,Lem75}
to problem \eqnok{prob_eig}. By making use of the specific structure of 
problem \eqnok{prob_eig}, each iteration of this method
requires the solution of a quadratic semidefinite programming
problem. It should also be noted that there are no complexity results available for the aforementioned 
spectral bundle method.  

Since problem \eqnok{prob_eig} can also be written as a bilinear saddle point problem:
\beq \label{prob_eig_sm}
\min_{x \in X} \left\{\lambda_1({\cal A}(x) =  
\max_{y \in Y} \langle  {\cal A}(x), y\rangle\right\},
\eeq
where $Y := \left\{y \in \bbr^{m \times m} | \mbox{Tr}(y) = 1, y \succeq 0\right\}$,
we can apply Nesterov's smoothing scheme (NEST-S)~\cite{Nest05-1,Nest06-1} and
the USL method developed in Section~\ref{sec-saddle} for
solving \eqnok{prob_eig_sm}. These methods can find an $\epsilon$-solution 
of \eqnok{prob_eig_sm} in at most ${\cal O}(1/\epsilon)$ iterations. 
It should be noted that the iteration costs of NEST-S and USL can slightly differ from each other. 
More specifically, USL applied to \eqnok{prob_eig_sm}
requires a full eigenvalue decomposition and computation of the adjoint operator ${\cal A}^*$ 
to define $h_\eta$ (see \eqnok{lower_saddle}) in step 1 of procedure~$\USLGAP$.
In addition, it requires to find
a maximum eigenvalue to compute $f(x_k^u)$ in step 3a 
of procedure~$\USLGAP$. On the other hand, each iteration of
NEST-S requires two (or one in some variants of Nesterov's method, see, e.g., \cite{Lan10-3}) full eigenvalue 
decompositions and computations of the adjoint operator ${\cal A}^*$. 

Our goal is to compare the four different algorithms,
namely: NERML, APL, USL and NEST-S, applied to solve problems in the form of \eqnok{prob_eig}
or \eqnok{prob_eig_sm}.
More details about the implementation of these algorithms
are as follows.
\begin{itemize}
\item {\sl Prox-functions.} 
If the feasible set $X$ is
a standard simplex given by $\left\{ x \in \bbr^n | \sum_{i=1}^n x_i = 1, x_i \ge 0, \forall i\right\}$,
the prox-function of $X$, as required by all these four algorithms,
is set to $\w(x) = \sum_{i=1}^n x_i \log x_i$ and the norm is set to $\|x\|_1$.
If $X$ is a box, then we set $\w(x) = \|x\|_2^2/2$ and the norm is set to $\|x\|_2$.
The prox-function of $Y$, as required by the algorithms  USL and NEST-S,
is set to $v(y) = \sum_{i=1}^n \lambda_i(y) \log \lambda_i(y)$, and
the norm is set to $\sum_{i=1}^n \lambda_i(y)$, where $\lambda_i(y)$, $i = 1, \ldots, n$,
denote the eigenvalues of $y \in Y$. Under this setting, the value of ${\cal D}_{v,Y}$
can be bounded by $\ln m$. Hence, we can set $Q_1 = \ln m$ in our implementation of the USL method.
\item {\sl Localizers.} For the APL, USL and NERML algorithms, we define the localizer $X'_k$ as
\[
X'_{k} = \left\{x \in X: \langle  
\nabla \w(x^l_k), x - x_{k}\rangle \ge 0 \right\} \bigcap {\cal M}_k, \ \
k \ge 1,
\]
where ${\cal M}_k$ denotes the intersection of totally at most $B$ half spaces of 
the form $\{x: h(x^l_k,x) \le l \}$ which have been generated most recently. 
Note that the larger the value of $B$
is, the more difficult the subproblems of BL type methods (e.g., \eqnok{def_lb} and \eqnok{prox_step}) are.
On the other hand, a larger $B$ might help to compute a better lower bound in~\eqnok{def_lb}.
It is found from our initial experiments that different values of $B$ within $[10, 30]$ perform
almost equally well. 
\item {\sl Subproblems for BL type methods.}
If the problem dimension $n$ is relatively small, say $n \le 5000$, we can solve 
the subproblems for BL type methods (e.g., \eqnok{def_lb} and \eqnok{prox_step}) by Mosek~\cite{Mosek}.
In this case we set the bundle limit $B$ to $30$. 
If the size of $n$ is very big, it will be time-consuming to directly solve the subproblems of BL type methods.
However, observing that the number of constraints in these subproblems
is very small (at most $B+1$), one can conveniently solve the Lagrangian duals of these subproblems
(see Ben-tal and Nemirovski~\cite{BenNem05-1}). In particular, if $n$ is big, say $n \ge 5000$, we set $B =10$
and solve the Lagrangian dual of these subproblems by using the ABL method, which can solve efficiently
small dimensional CP problems, similarly to the BL method~\cite{LNN,BenNem05-1}.
\item  {\sl Fine-tuning for NEST-S.} For the NEST-S scheme, we compute the Lipschitz constant 
${\cal L}_\eta$ by \eqnok{new_ls}, where the smoothing parameter
$\eta$ is set to 
\[
\frac{2 \|\cal A\|}{N+1} \frac{D_{\w,X}}{\sigma_\w \sigma_v D_{v,Y}}.
\] 
and the operator norm $\|\cal A\|$ is computed according to \cite{Nest06-1}. 
Note however, that the resulting estimate of ${\cal L}_\eta$ can be rather
conservative, which leads to the slow convergence of the NEST-S scheme. 
We had also implemented a variant of Nesterov's method
which can adaptively search for the Lipschitz constant ${\cal L}_\eta$ (\cite{Nest07-1}).
However, our preliminary experiments indicate that the improvement from
this approach is not significant. In our final experiments, we run 
NEST-S four times, and each time we multiply the 
Lipschitz constant ${\cal L}_\eta$ estimated above by a different factor: $10^{-1}$,  
$10^{-2}$,  $10^{-3}$, or $10^{-4}$. We then report the best solutions,
in terms of the objective value, obtained from these four runs of 
the NEST-S scheme.
\item {\sl Others.} We set $\beta = \theta=1/2$, and specify 
$\{\alpha_k\}$ according to Proposition~\ref{APL-step}.a) in
the APL and USL methods. All the codes are implemented in MATLAB2007 under Windows Vista and 
the experiments were conducted on an INTEL 2.53 GHz labtop.
\end{itemize}

Our first experiments were conducted on a set of randomly generated SDP
instances, each of which has  
various sizes of $A_i$'s for $i=1, \ldots,n$.  
We also assume that the feasible set $X$ is given by
a standard simplex. More details about 
these instances are shown in Table~\ref{table-eig}, where
$n$ is the dimension of $x$, $\ub_1$ is the objective value
at $p_1 = (1/n, \ldots, 1/n)$, and $\Delta_1 = \ub_1 - \lb_1$ 
denotes the initial gap with $\lb_1$ given by \eqnok{ini_lb_saddle}. 
 We run $200$ iterations for the four algorithms mentioned above
and  report the objective values obtained at the $100$th and $200$th iteration
in column $2$ and column $4$, respectively. For the APL, USL and NERML
algorithms, we also report the optimality gap $\Delta_s$ at the $100$th and
$200$th iteration, respectively, in column $3$ and $5$ of Table \ref{table-eig}.
The CPU time (in seconds) for running these algorithms is reported
in column $6$ of Table \ref{table-eig}. It should be noted that
we only report the CPU time for one run of the NEST-S algorithm, although
we had run it for $4$ times to find a good estimate of ${\cal L}_\eta$.

\begin{table}[h] 
\caption{Experiments with random SDP problems}
\centering \label{table-eig}
\begin{tabular}{|l|l l |l l|l|}
\hline
\multicolumn{6}{|c|}{E1: $n=1,000$, $m=400$, $d=2\%$, $\mbox{ub}_1=6.329960$, $\Delta_1=5.08e-1$}\\
\hline
\multicolumn{1}{|l|}{alg.} 
& \multicolumn{1}{|l}{$\mbox{ub}_{100}$} 
& \multicolumn{1}{l}{$\Delta_{100}$} 
& \multicolumn{1}{|l}{$\mbox{ub}_{200}$} 
& \multicolumn{1}{l|}{$\Delta_{200}$} 
& \multicolumn{1}{l|}{Time}\\
\hline
USL &$6.026060$ & $8.48e-5$ & $6.026045$& $1.28e-6$ &$174.10$\\
NEST-S& $6.077040$ &-& $6.076351$ & -& $190.51$\\
APL & $6.026048$ &$4.23e-5$ & $6.026045$ &$1.22e-6$ &$109.59$\\
NERML & $6.026323$ & $7.97e-4$ &$6.026084$ & $2.39e-4$ & $100.37$ \\
\hline
\hline
\multicolumn{6}{|c|}{E2: $n=1,000$, $m=600$,  $d=2\%$, $\mbox{ub}_1=7.788735$, $\Delta_1=6.45e-1$}\\
\hline
\multicolumn{1}{|l|}{alg.} 
& \multicolumn{1}{|l}{$\mbox{ub}_{100}$} 
& \multicolumn{1}{l}{$\Delta_{100}$} 
& \multicolumn{1}{|l}{$\mbox{ub}_{200}$} 
& \multicolumn{1}{l|}{$\Delta_{200}$} 
& \multicolumn{1}{l|}{Time}\\
\hline
USL &$7.458582$ & $3.39e-4$ & $7.458538$& $5.07e-6$ &$364.93$\\
NEST-S& $7.539811$ &-& $7.538583$ & -& $552.12$\\
APL & $7.458561$ &$8.96e-5$ & $7.458537$ & $1.96e-6$ & $166.27$\\
NERML& $7.458801$ &$1.07e-3$ & $7.458557$ &$8.35e-5$ & $142.16$\\
\hline
\hline
\multicolumn{6}{|c|}{E3: $n=1,000$, $m=800$, $d=2\%$, $\mbox{ub}_1=8.855385$, $\Delta_1=5.57e-1$}\\
\hline
\multicolumn{1}{|l|}{alg.} 
& \multicolumn{1}{|l}{$\mbox{ub}_{100}$} 
& \multicolumn{1}{l}{$\Delta_{100}$} 
& \multicolumn{1}{|l}{$\mbox{ub}_{200}$} 
& \multicolumn{1}{l|}{$\Delta_{200}$} 
& \multicolumn{1}{l|}{Time}\\
\hline
USL &$8.555496$ & $1.35e-4$ & $8.555475$& $4.08e-6$ &$799.30$\\
NEST-S & $8.635632$ &-& $8.635473$ & -& $1347.26$\\
APL &$8.555484$ & $6.61e-5$ & $8.555475$ & $2.05e-6$ & $275.22$\\
NERML & $8.555743$ &$7.45e-4$ & $8.555494$ & $1.24e-4$ & $213.25$ \\ 
\hline
\end{tabular}
\end{table}

We can draw a few conclusions from our experiments with these 
random SDP instances. Firstly, among the two methods
with ${\cal O}(1/\epsilon)$ convergence, USL 
can significantly outperform NEST-S: the former algorithm can
reach $6$ accuracy digits after $200$ iterations while the latter
algorithm reaches at most $2$ accuracy digits for these instances. 
Secondly, for the two nonsmooth methods, APL consistently
outperforms NERML in solution quality while
the computational time is comparable to the latter one.
Thirdly, while the solution quality of the USL method is significantly
better than the one of the NERML algorithm, it is interesting to notice
that the solution quality of the APL algorithm is comparable or better than
that of the USL. One plausible explanation is that
the problems to be solved, due to the inherent randomness, 
are smooth along most part of the trajectory of the APL algorithm.

\vgap

Our second experiments were carried out for a class of
more structured SDP instances, namely a class of Lovasz capacity problems. 
Let $(N,E)$ denote a graph with $m$ nodes in $N$ and $n$ edges in $E$.
The Lovasz capacity $\lov$ of $(N,E)$ is defined by
\beq \label{def_lov}
\lov := \min_{x \in \cX} \left\{\Phi(x) := \lambda_1\left(d + x\right)\right\}.
\eeq 
Here $\cX := \left\{x \in {\cal S}^m: x_{ij} = 0 \mbox{ if } (i,j) \notin E \right\}$,
${\cal S}^m$ denotes the set of symmetric matrices in $\bbr^{m \times m}$ and
$d$ is a $m \times m$ constant matrix given by
\[
d_{ij} := 
\left\{
\begin{array}{ll}
0, (i,j) \in E, \\
1, (i,j) \notin E.
\end{array}
\right.
\]
Note that for an optimal $x$ of problem \eqnok{def_lov} the matrix $\lov I - (d + x)$
is positive semidefinite, so that nonzero entries in $x$
satisfy $|x_{ij}| \le \lov -1$. It follows that if
$v$ is a valid upper bound on $\lov$, then
problem \eqnok{def_lov} is equivalent to
\beq \label{def_lov1}
\lov = \min_{x \in \cX_v}  \Phi(x),
\eeq
where $\cX_v := \{ x \in \cX: |x_{ij}| \le v -1 \}$. In view of this observation,
we incorporate following enhancement into
the aforementioned first-order methods (NERML, APL, USL and NEST-S) applied to \eqnok{def_lov1}:
for all these methods, we update the upper bound $v$ in defining 
the feasible set $\cX_v$ from time to time. 
In particular, we update the value of $v$
in each phase of BL type methods. For NEST-S, we update
the value of $v$ whenever a new upper bound on $\lov$ becomes available 
(as noted by \cite{tseng08-1}, the optimal convergence of Nesterov's method
will be guaranteed with such a domain shrinking strategy). 

We generate a set of random graph instances as follows.
For a given number of nodes $m$ and a designed number of edges $\bar n$,
we first generate $m-1$ edges, each one connecting
a new node with a randomly selected existing node.
After that, we create $\bar n - m +1$ random edges and remove
those redundant edges. Thus, the actual number of edges $n$ can be smaller
than the designed number $\bar n$. Totally $6$ instances have been generated in this manner and 
the number of edges $n$ (and hence the number of decision variables) ranges from $11,503$ 
to $77,213$ (see Table~\ref{table-graph}). We also report the initial objective values
of these instances at $x_0=0$ in column $4$ of Table~\ref{table-graph}.
In order to compare the aforementioned algorithms for computing Lovasz capacity, 
we first run the NERML algorithm for $1,000$ iterations and record the quality of the output solutions
in terms of the generated upper bound. We then terminate
the remaining three algorithms, namely APL, USL and NEST-S,
whenever similar solution quality is achieved or the $1,000$
iteration limit is reached. We report the number of iterations, the computed upper bound
and CPU time in columns $2-4$, $5-7$, $8-9$ and $10-12$, respectively, for NERML, APL, USL and NEST-S
in Table~\ref{table-Lovasz}. From these results, we can safely draw the following conclusions.
Firstly, all the BL type methods significantly outperform NEST-S for these Lovasz capacity instances.
Secondly, while APL significantly outperforms NERML for the bigger instances G52, G61 and G62,
the USL method, which combines the advantages of both APL and NEST-S,
can significantly outperform NEST-S, NERML and APL for all these Lovasz capacity instances.
\begin{table}[h] 
\caption{Lovasz capacity instances}
\centering \label{table-graph}
\begin{tabular}{|l|l l l|}
\hline
\multicolumn{1}{|l|}{Inst.} 
& \multicolumn{1}{|l}{$m$} 
& \multicolumn{1}{l}{$n$}  
& \multicolumn{1}{l|}{$\Phi(0)$}\\
\hline
G41 & $400$ & $11,503$  &$342.62$   \\
G42 & $400$ & $21,692$ &$291.78$  \\
G51 & $500$ & $23,076$ &$407.88$   \\
G52 & $500$ & $47,910$  &$308.72$  \\
G61 & $600$ & $51,429$  &$428.87$  \\
G62 & $600$ & $77,213$  & $343.10$  \\
\hline
\end{tabular}
\end{table}


\begin{table}[h] 
\caption{Comparison of first-order methods for Lovasz capacity instances}
\centering \label{table-Lovasz}
\begin{tabular}{|l|l l l | l l l | l l l |l l l |}
\hline
\multicolumn{1}{|l|}{}
& \multicolumn{3}{c|}{NERML}
& \multicolumn{3}{c|}{APL}
& \multicolumn{3}{c|}{USL}
& \multicolumn{3}{c|}{NEST-S}\\
\multicolumn{1}{|c|}{Inst.}&
\multicolumn{1}{c}{Iter.} & \multicolumn{1}{c}{$\ub$}&
\multicolumn{1}{c|}{Time} &
\multicolumn{1}{c}{Iter.} & \multicolumn{1}{c}{$\ub$}&
\multicolumn{1}{c|}{Time} &
\multicolumn{1}{c}{Iter.} & \multicolumn{1}{c}{$\ub$}&
\multicolumn{1}{c|}{Time} &
\multicolumn{1}{c}{Iter.} & \multicolumn{1}{c}{$\ub$}&
\multicolumn{1}{c|}{Time}  \\
\hline
G41 & 1,000 & 63.88 & 829.58 & 800 & 63.61& 890.18 & 40& 62.79 & 44.08 &1,000 &342.31&980.29\\
G42 & 1,000 & 41.26 & 963.82 & 1,000&41.33 & 1166.17 & 30& 41.19 & 37.89 &1,000& 291.43&1078.53\\
G51 & 1,000 & 62.82 & 1292.38 & 1,000&63.21 & 1697.79 &30& 61.13 & 54.56&1,000&407.68 & 1964.73\\
G52 & 1,000 & 42.19 & 1724.29 & 70 & 41.50 &166.54 & 20 & 38.72& 45.30&1,000& 308.46&2403.18\\
G61 & 1,000 & 68.53 & 2343.86 & 20& 66.76 & 60.46&10& 59.30 & 32.67&1,000 &428.72& 3948.19\\
G62 & 1,000 & 40.43 & 4062.83 & 110&40.36 &455.54 &20& 39.47 & 96.89&1,000&342.91 &4309.73\\
\hline
\end{tabular}
\end{table}


\subsection{A class of two-stage stochastic programming problems} \label{sec-sp}
In this subsection, we consider the classic two-stage stochastic linear programming
given by
\beq \label{prob-sp}
\min_{x \in X} \left\{f(x) = c^T x + \bbe[V(x,\xi)] \right\},
\eeq
with
\beq \label{def_v}
V(x,\xi) = \min \left\{q^T \pi: W \pi = h + T x, \pi \ge 0 \right\}.
\eeq
Here, $x \in \bbr^{n_1}$ and $\pi \in \bbr^{n_2}$, respectively, are the
first and second-stage decision variables, $X \subseteq \bbr^{n_1}$
is a nonempty convex compact set, and ${\mathbb \xi \equiv (q, h, T)}$ is a random vector
with a known distribution supported on $\Xi \subseteq \bbr^{n_2 + m_2 + m_2\times n_1}$.
We assume that problem \eqnok{def_v} is feasible for every possible realization of $\xi$, i.e.,
problem \eqnok{prob-sp} has a complete recourse.
Moreover, for the purpose of illustrating the effectiveness of the algorithms developed in
this paper, we assume that $\xi$ is a discrete random vector and the number of
possible realizations of $\xi$ (or the sample space) is not too big.

It should be noted that if $\xi$ is a continuous random vector or the number of
possible realizations of $\xi$ is astronomically large, to solve problem \eqnok{prob-sp}
is highly challenging, due to the fundamental difficulty of computing the expectation 
to a high accuracy when the dimension of $\xi$ is high, see \cite{NJLS09-1,lns11} for 
a discussion on some recent advancements in this area. However, if the number 
of possible realizations of $\xi$ is not astronomically large, it is possible 
to solve problem \eqnok{prob-sp} to high accuracy 
in a reasonable amount of time
by using more powerful algorithms. This is indeed what we intend to demonstrate 
in this subsection. 
 
Since problem \eqnok{prob-sp} is nonsmooth in general, 
one can apply the NERML or APL methods.
These methods, in the worst case,
require ${\cal O} ( 1/ \epsilon^2)$ iterations to find an $\epsilon$-solution
of problem~\eqnok{prob-sp}. 
Recently, Ahmed~\cite{Ahmed06-1} noted that one can improve the complexity bound for solving
\eqnok{prob-sp} to ${\cal O}(1/\epsilon)$, by applying Nesterov's smoothing scheme to \eqnok{prob-sp}.
The basic idea is as follows. Let ${\cal Y}(q) := \{W^T y \le q\}$ and ${\cal B}_{m_2}$ be
the Euclidean ball in $\bbr^{m_2}$.
Note that by strong duality, we have
\beq \label{def_dual_v}
V(x, \xi) = \max \left\{ (h + Tx)^T y: y \in {\cal Y}(q) \right\}.
\eeq
Moreover, by Hoffman's Lemma~\cite{Hoffman52}, there exists a constant ${\cal R}_W > 0$ 
depending on $W$ such that
\[
{\cal Y}(q) \subseteq {\cal Y}(0) + {\cal R}_W \|q\| {\cal B}_{m_2}
= {\cal R}_W \|q\| {\cal B}_{m_2},
\]
where the last identity follows from the fact that ${\cal Y}(0) = \{0\}$
due to the complete recourse assumption. In other words, the
feasible region of \eqnok{def_dual_v} is bounded.
We can then uniformly approximate $f(x)$ in \eqnok{prob-sp} by
$
f_\eta(x) := c^T x + \bbe[V_\eta(x,\xi)]
$
for some $\eta > 0$, where
\beq \label{smoothed_sp}
V_\eta(x, \xi) = \max \left\{ (h + Tx)^T y - \eta \|y\|^2/2: y \in {\cal Y}(q)\right\}.
\eeq
However, the implementation of Nesterov's smoothing scheme is difficult,
since it is necessary to fine-tune a large number of problem parameters, including 
${\cal R}_W$, $\|q\|$ and $\|T\|$, as well 
as ${\cal D}_{\w,X}$. Due to the lack of good estimations for these parameters, especially, ${\cal R}_W$, 
no computational results have been reported in \cite{Ahmed06-1}.


In our experiments, we have implemented three methods, namely: APL and NERML
and USL, applied to problem \eqnok{prob-sp}. All these methods do not
require the input of any problem parameters and the implementation details are
similar to those in Subsection~\ref{sec-sdp}. We conduct our experiments on a few SP instances 
which have been studied by a few authors, namely: a telecommunication design (SSN) problem 
of Sen, Doverspike, and Cosares~\cite{sdc94} and the motor freight
carrier routing problem (20-term) of Mak, Morton, and Wood~\cite{mak:99}.
The dimensions of these instances are shown in Table \ref{table-sp-inst}, please
see \cite{lsw} for more details about these instances. 

\begin{table}[h] 
\caption{Dimension of the SP instances}
\centering \label{table-sp-inst}
\begin{tabular}{|l|l l |l l|}
\hline
\multicolumn{1}{|l|}{} 
& \multicolumn{1}{|l}{$n_1$} 
& \multicolumn{1}{l}{$m_1$} 
& \multicolumn{1}{|l}{$n_2$} 
& \multicolumn{1}{l|}{$m_2$} \\
\hline
SSN & $89$ & $1$ & $796$ &  $175$\\
20-term & $63$ & $3$ & $764$ & $124$\\
\hline
\end{tabular}
\end{table}

It is worth noting that here we assume that the number of possible 
realizations are fixed ($N=50$ or $100$) and hence obtain
four different instances, namely: SSN(50), SSN(100), 20-term(50) and
20-term(100). Noting that the initial optimality gap for these instances
are rather high (in order of $10^3$ or $10^7$), we run each algorithm
for $400$ iterations and the results are reported in Table~\ref{table-sp}.
The structure of the table is similar to Table~\ref{table-eig}
(see Subsection~\ref{sec-sdp}). We also run NERML for $1,000$ iterations
first and then check whether APL and USL can achieve similar gap reduction.
The latter results are reported in Table~\ref{table-sp-comp}. 

We can make a few observations from the numerical results in Tables~\ref{table-sp}
and \ref{table-sp-comp}.
Firstly, the iteration cost of the USL method is larger than that of the APL method,
which, in turn, is larger than that of the NERML algorithm. In particular,
the major iteration cost of the NERML and APL algorithm consists of solving
$N$ and $2N$ second-stage LP problems respectively, while the one of the 
USL algorithm involves the solutions of $N$ smoothed quadratic programming problems  
(see \eqnok{smoothed_sp}).
Secondly, both the APL and USL methods can significantly outperform the
the NERML algorithm in terms of the solution quality. As we can see from 
Tables \ref{table-sp} and Table \ref{table-sp-comp}, the NERML algorithm makes little progresses after $200$ 
iterations for these SP instances. Thirdly, the solution quality of the APL method is worse than that of
the USL method for solving the first two instances: SSN(50) and SSN(100),
but it significantly outperforms the latter one for solving the
last two instances: 20-term(50) and 20-term(100). One possible reason
is that the sizes of ${\cal D}_{v,Y} (\approx {\cal R}_W^2 \|q\|^2)$ for 
the last two instances are significantly larger than those for the first 
two instances, see Table~\ref{table-est} for the estimates on ${\cal D}_{v,Y}$ 
reported by the USL method (with $Q_1 = 1$).

\begin{table}[h] 
\caption{Experiments with the SP instances}
\centering \label{table-sp}
\begin{tabular}{|l|l l |l l|l|}
\hline
\multicolumn{6}{|c|}{SSN(50): $\mbox{ub}_1=2.352586e+2$, $\Delta_1=3.923265e+3$ }\\
\hline
\multicolumn{1}{|l|}{alg.} 
& \multicolumn{1}{|l}{$\mbox{ub}_{200}$} 
& \multicolumn{1}{l}{$\Delta_{200}$} 
& \multicolumn{1}{|l}{$\mbox{ub}_{400}$} 
& \multicolumn{1}{l|}{$\Delta_{400}$} 
& \multicolumn{1}{l|}{Time}\\
\hline
APL &  $4.839075$ &$2.437395e-3$ & $4.838074$ & $5.053628e-7$&$366.30$ \\
USL &  $4.838125$ & $1.837968e-4$ & $4.838073$ & $6.599449e-7$ & $754.31$\\
NERML & $5.550903$ & $5.012402e+0$ & $5.086603$ & $1.303485e+0$ &$193.47$ \\
\hline
\hline
\multicolumn{6}{|c|}{SSN(100): $\mbox{ub}_1=2.407279e+2$, $\Delta_1=4.023982e+3$ }\\
\hline
\multicolumn{1}{|l|}{alg.} 
& \multicolumn{1}{|l}{$\mbox{ub}_{200}$} 
& \multicolumn{1}{l}{$\Delta_{200}$} 
& \multicolumn{1}{|l}{$\mbox{ub}_{400}$} 
& \multicolumn{1}{l|}{$\Delta_{400}$} 
& \multicolumn{1}{l|}{Time}\\
\hline
APL & $7.354770$ &$9.148243e-3$ & $7.352610$ & $4.198017e-6$ & $730.15$ \\
USL & $7.354090$ & $2.683424e-3$ & $7.352610$ & $6.804606e-7$ & $1471.62$\\
NERML &$8.323381$ & $4.771295e+0$ &$7.578802$ & $1.079491e+0$ & $383.27$ \\
\hline
\hline
\multicolumn{6}{|c|}{20-term(50): $\mbox{ub}_1=7.718543e+5$, $\Delta_1=1.804693e+7$ }\\
\hline
\multicolumn{1}{|l|}{alg.} 
& \multicolumn{1}{|l}{$\mbox{ub}_{200}$} 
& \multicolumn{1}{l}{$\Delta_{200}$}
& \multicolumn{1}{|l}{$\mbox{ub}_{400}$} 
& \multicolumn{1}{l|}{$\Delta_{400}$} 
& \multicolumn{1}{l|}{Time}\\
\hline
APL & $2.549453e+5$ &$1.229655e-3$ & $2.549453e+5$ &$2.405432e-7$ & $1056.82$ \\
USL & $2.551031e+5$ & $1.896133e+3$ & $2.549602e+5$ & $3.310795e+2$ & $1209.53$\\
NERML & $2.587140e+5$ & $1.473815e+4$ & $2.576649e+5$& $1.368899e+4$ & $301.03$ \\
\hline
\hline
\multicolumn{6}{|c|}{20-term(100): $\mbox{ub}_1=7.664067e+5$, $\Delta_1=1.801832e+7$ }\\
\hline
\multicolumn{1}{|l|}{alg.} 
& \multicolumn{1}{|l}{$\mbox{ub}_{200}$} 
& \multicolumn{1}{l}{$\Delta_{200}$} 
& \multicolumn{1}{|l}{$\mbox{ub}_{400}$} 
& \multicolumn{1}{l|}{$\Delta_{400}$} 
& \multicolumn{1}{l|}{Time}\\
\hline
APL & $2.532875e+5$ &$3.679608e-3$ & $2.532875e+5$ &$2.463930e-7$ & $1895.63$ \\
USL & $2.533441e+5$ &$5.119095e+2$ & $2.532923e+5$ & $7.614912e+1$ & $2517.26$\\
NERML &$2.581546e+5$ & $2.171689e+4$ & $2.540804e+5$ & $3.754735e+3$ & $602.60$ \\
\hline
\end{tabular}
\end{table}

\begin{table}[h] 
\caption{Comparison of level methods for the SP instances}
\centering \label{table-sp-comp}
\begin{tabular}{|l|l l l | l l l | l l l |}
\hline
\multicolumn{1}{|l|}{}
& \multicolumn{3}{c|}{NERML}
& \multicolumn{3}{c|}{APL}
& \multicolumn{3}{c|}{USL}\\
\multicolumn{1}{|c|}{Inst.}&
\multicolumn{1}{c}{Iter} & \multicolumn{1}{c}{gap}&
\multicolumn{1}{c|}{Time} &
\multicolumn{1}{c}{Iter} & \multicolumn{1}{c}{gap}&
\multicolumn{1}{c|}{Time} &
\multicolumn{1}{c}{Iter} & \multicolumn{1}{c}{gap}&
\multicolumn{1}{c|}{Time} \\
\hline
SSN(50) & 1,000 & 4.689592e-1& 497.45&90 & 4.379673e-1& 85.67 &90 &2.935656e-1  &177.16 \\
SSN(100) & 1,000 & 1.001421e+0 &1037.11 &60 & 9.142773e-1&114.92 &60 &8.150025e-1 &226.91\\
20-term(50) & 1,000 &1.058791e-1 &794.87 &140 & 4.959911e-2& 301.40&590 &4.277621e-2 &1632.86 \\
20-term(100) & 1,000 &3.754735e+3 &1437.31 &70 &1.730272e+3 &197.45 &110 &1.740399e+3 &638.86\\
\hline
\end{tabular}
\end{table}

\begin{table}[h] 
\caption{Estimates on ${\cal D}_{v,Y}$}
\centering \label{table-est}
\begin{tabular}{|l|l l |l l|}
\hline
\multicolumn{1}{|l|}{} 
& \multicolumn{1}{|l}{SSN(50)} 
& \multicolumn{1}{l}{SSN(100)} 
& \multicolumn{1}{|l}{20-term(50)} 
& \multicolumn{1}{l|}{20-term(100)} \\
\hline
$Q$ & $64$ & $128$ & $1.68e+7$ &  $1.68e+7$\\
\hline
\end{tabular}
\end{table}



\setcounter{equation}{0} 
\section{Concluding remarks} \label{c_remarks}
In this paper, we present new bundle-level type methods for
convex programming. In particular, we show that both the ABL and
APL methods are uniformly optimal for solving smooth, nonsmooth and
weakly smooth problems without requiring the input of
any smoothness information. We also demonstrate that, with little modification,
the APL method is optimal for solving a class of composite CP
problems. Based on the APL method, we develop a new smoothing technique,
namely the USL method, which can achieve the optimal
complexity for solving a class of saddle point problems
without requiring the input of any problem parameters.
We demonstrate the significant advantages of the APL and USL methods
over some existing first-order methods for solving
certain classes of semidefinite programing and stochastic programming 
problems. 

\vgap

{\bf Acknowledgement:}
The author is very grateful to the co-editor Professor Adrian Lewis,
the associate editor and two anonymous referees
for their very useful suggestions for improving the quality and exposition
of the paper. 

\setcounter{equation}{0}
\section{Appendix}

In this section, we provide the proof of Lemma~\ref{tech_result}.

Let $F$ and $F_\eta$ be defined in \eqnok{nonsmooth1} and \eqnok{sm-approx},
respectively. Also let us denote, for any $\eta > 0$ and $x \in X$,
\beq \label{def_psi}
\psi_x(z) := F_\eta(x) + \langle \nabla F_\eta(x), z - x \rangle
+ \frac{{\cal L}_\eta}{2} \|z - x\|^2 + \eta {\cal D}_v,
\eeq
where ${\cal D}_v$ and ${\cal L}_\eta$ are defined in \eqnok{def_cal_DX} and
\eqnok{new_ls}, respectively. Clearly, in view of \eqnok{smoothness} and
\eqnok{closeness}, $\psi_x$ is a majorant of both $F_\eta$ and $f$.
Also let us define
\beq \label{def_Zx}
Z_x := \left\{ z \in \bbr^n: \|z - x\|^2 = \frac{2}{{\cal L}_\eta}
\left[ \eta {\cal D}_v + F_\eta(x) - F(x) \right] \right\}. 
\eeq
Clearly, by the first relation in \eqnok{closeness}, we have
\beq \label{bnd_Zx}
\|z-x\|^2 \le \frac{2 \eta {\cal D}_v}{{\cal L}_\eta},  
\ \ \forall \, z \in Z_x.
\eeq
Moreover, we can easily check that, for any $z \in Z_x$, 
\beq \label{app_eqa}
\psi_x(z) + \langle \nabla \psi_x(z), x - z \rangle = F(x),
\eeq
where $\nabla \psi_x(z) = \nabla F_\eta(x) + {\cal L}_\eta (z-x)$.

\vgap

The following results provides the characterization of a subgradient
direction of $F$.

\begin{lemma} \label{app_tech1}
Let $x \in \bbr^n$ and $p \in \bbr^n$ be given.
Then, $\exists z \in Z_x$ such that
\[
\langle F'(x), p \rangle \le \langle \nabla \psi_x(z), p\rangle = 
\langle \nabla F_\eta(x) + {\cal L}_\eta (z-x), p \rangle.
\]
where $F'(x) \in \partial F(x)$.
\end{lemma}

\begin{proof}
Let us denote
\[
t = \frac{1}{\|p\|}\left\{
\frac{2}{{\cal L}_\eta}
\left[ \eta {\cal D}_v + F_\eta(x) - F(x) \right]
\right\}^\frac{1}{2}
\]
and $z_0 = x + t p$. Clearly, in view of \eqnok{def_Zx},
we have $z_0 \in Z_x$. By convexity of $F$ and \eqnok{app_eqa},
we have
\beqas
F(x) + \langle F'(x), tp \rangle \le F(x + tp) &=& \psi_x(z_0)
= F(x) + \langle \nabla \psi_x(z_0), z_0 - x\rangle \\
&=& F(x) + t \langle \nabla \psi_x(z_0), p\rangle,
\eeqas
which clearly implies the result.
\end{proof}

\vgap

We are now ready to prove Lemma~\ref{tech_result}.

\noindent{\bf Proof of Lemma~\ref{tech_result}}.
First note that by the convexity of $F$, we have
\[
F(x_0) - \left [F(x_1) + \langle F'(x_1), x_0 - x_1\right] \rangle 
\le \langle F'(x_0), x_0 - x_1 \rangle  + \langle F'(x_1), x_1 - x_0 \rangle.
\]
Moreover, by Lemma~\ref{app_tech1}, $\exists z_0 \in Z_{x_0}$ and $z_1 \in Z_{x_1}$ s.t.
\beqas
\lefteqn{\langle F'(x_0), x_0 - x_1 \rangle  + \langle F'(x_1), x_1 - x_0 \rangle }\\
& \le & \langle \nabla F_\eta(x_0) - \nabla F_\eta(x_1),
x_0 - x_1 \rangle + {\cal L}_\eta \langle z_0 - x_0 - (z_1 - x_1), x_0 - x_1\rangle\\
&\le& {\cal L}_\eta \|x_0 - x_1\|^2 + {\cal L}_\eta (\|z_0 - x_0\| + \|z_1-x_1\|) \|x_0 - x_1\|\\
&\le& {\cal L}_\eta \|x_0 - x_1\|^2 + 2 {\cal L}_\eta 
\left(\frac{2 \eta {\cal D}_v}{{\cal L}_\eta} \right)^\frac{1}{2} \|x_0 - x_1\| \\
&=& 
\frac{\|A\|^2}{\sigma_v \eta} \|x_0 - x_1\|^2 + 2 
\left( \frac{2 \|A\|^2 {\cal D}_v}{\sigma_v}\right)^\frac{1}{2} \|x_0 - x_1\|,
\eeqas
where the last inequality and equality follow from \eqnok{bnd_Zx} and \eqnok{new_ls},
respectively.
Combining the above two relations, we have
\[
F(x_0) - \left[F(x_1) + \langle F'(x_1), x_0 - x_1 \rangle\right]
\le \frac{\|A\|^2}{\sigma_v \eta} \|x_0 - x_1\|^2 + 
2 \left( \frac{2 \|A\|^2 {\cal D}_v}{\sigma_v}\right)^\frac{1}{2} \|x_0 - x_1\|.
\]
The result now follows by tending $\eta$ to $+\infty$ in the above relation.
\endproof

\bibliographystyle{plain}
\bibliography{../glan-bib}
\end{document}